\documentclass[11pt]{article}
\usepackage{latexsym,amsmath,color,amsthm,amssymb,epsfig,graphicx,mathrsfs}
\usepackage{graphicx}
\usepackage{amssymb}
\usepackage[left=1in,top=1in,right=1in,bottom=1in]{geometry}
\usepackage[linktocpage=true]{hyperref}
\usepackage{setspace}
\usepackage{amssymb, amsmath, amsthm, graphicx,mathrsfs}
\usepackage{caption}

\usepackage{comment}
\def\qed{\hfill\ifhmode\unskip\nobreak\fi\quad\ifmmode\Box\else\hfill$\Box$\fi}
\def\ite#1{\hfill\break${}$\hbox to 50pt {\quad(#1)\hfill}}

\def\cG{{\mathcal G}}
\def\cH{{\mathcal H}}

\newtheorem{thm}{Theorem}[section]

\newtheorem{const}[thm]{Construction}
\newtheorem{definition}[thm]{Definition}
\newtheorem{remark}[thm]{Remark}
\newtheorem{lem}[thm]{Lemma}
\newtheorem{conj}[thm]{Conjecture}

\newtheorem{question}[thm]{Question}

\usepackage[parfill]{parskip} % this is a minor improvement on the two commented-out commands that follow, which handles paragraph spacing better in unusual cases.

\begin{document}

%\pagestyle{myheadings} 
%\markright{{\small{\sc A.~Kostochka,  R. Luo, and D. Zirlin:   Super-pancyclic hypergraphs}}}

\title{\vspace{-0.5in} Longest cycles in 3-connected  hypergraphs and  bipartite graphs }

\author{
{{Alexandr Kostochka}}\thanks{
\footnotesize {University of Illinois at Urbana--Champaign, Urbana, IL 61801
 and Sobolev Institute of Mathematics, Novosibirsk 630090, Russia. E-mail: \texttt {kostochk@math.uiuc.edu}.
 Research %%% of this author
is supported in part by NSF grant  DMS-1600592
and grants  18-01-00353A and 19-01-00682
  of the Russian Foundation for Basic Research.
}}
  \and{Mikhail Lavrov\thanks{Department of Mathematics, University of Illinois at Urbana--Champaign, IL, USA,  mlavrov@illinois.edu.}}
\and{{Ruth Luo}}\thanks{University of Califonia, San Diego, La Jolla, CA 92093, USA and University of Illinois at Urbana--Champaign, Urbana, IL 61801, USA. E-mail: {\tt ruluo@ucsd.edu}.
Research %%% of this author
is supported in part by NSF grants  DMS-1600592 and DMS-1902808.
}
\and{{Dara Zirlin}}\thanks{University of Illinois at Urbana--Champaign, Urbana, IL 61801, USA. E-mail: {\tt zirlin2@illinois.edu}.
Research %%% of this author
is supported in part by Arnold O. Beckman Research Award (UIUC) RB20003.}
}

\date{\today}

\maketitle

\vspace{-0.3in}

\begin{abstract}
In the language of hypergraphs, our main result is a Dirac-type bound:
 we prove that every $3$-connected hypergraph $\cH$ with $ \delta(\cH)\geq \max\{|V(\cH)|, \frac{|E(\cH)|+10}{4}\}$
has a hamiltonian Berge cycle. 

This is sharp and refines a conjecture by Jackson from 1981 (in the language of bipartite graphs). 
Our proofs are  in the language of bipartite graphs, since the incidence graph of  each hypergraph is bipartite.

%We  mostly use the language of bipartite graphs, because every bipartite graph is the incidence graph of a multihypergraph. In particular, we develop some results of Jackson on the existence of long cycles in bipartite graphs where the vertices in one part have high minimum degree. %His results naturally transfer to hypergraphs by considering  their incidence  graphs (that are bipartite). 
%We prove a conjecture of Jackson from 1981 on long cycles in 2-connected bipartite graphs.

\medskip\noindent
{\bf{Mathematics Subject Classification:}}  05C35,   05C38,  05C65, 05D05.\\
{\bf{Keywords:}} Longest cycles, degree conditions, pancyclic hypergraphs.
\end{abstract}

\section{Introduction}

\subsection{Long cycles in bipartite graphs}

For positive integers $n, m,$ and $\delta$ with $\delta \leq m$, let $\cG(n,m,\delta)$ denote the set of all bipartite graphs with a partition $(X, Y)$ such that $|X| = n\geq 2, |Y|=m$ and for every $x \in X$, $d(x) \geq \delta$. In 1981, Jackson~\cite{jackson} proved that if  $ \delta\geq \max\{n,\frac{m+2}{2}\}$, then every graph $G\in \cG(n,m,\delta)$ contains a cycle of length $2n$, i.e., a cycle that covers $X$.
This result is sharp. 
 Jackson also conjectured that if  $G \in \cG(n,m,\delta)$ is 2-connected, then  the upper bound on $m$ can be weakened.

\begin{conj}[Jackson~\cite{jackson,jackson3}]\label{jacksonconj} Let $m,n,\delta$ be integers. If $\delta\geq \max\{n, \frac{m+5}{3}\}$,
then every $2$-connected graph $G \in \cG(n,m,\delta)$ contains a cycle of length $2n$.\end{conj} 

Recently, the conjecture was proved in~\cite{KLZ}. The restriction $\delta\geq \frac{m+5}{3}$ cannot be weakened because of the following example.

\begin{const}\label{con5}
Let $n_1 \geq n_2 \geq n_3\geq 1$ be such that $n_1 + n_2 + n_3 = n$.
 Let $G_3(n_1, n_2, n_3;\delta) \in \cG(n,3\delta-4, \delta)$ be the bipartite graph obtained from $K_{\delta-2, n_1} \cup K_{\delta-2, n_2} \cup K_{\delta-2, n_3}$  by adding two vertices $a$ and $b$ that are both adjacent to every vertex in the parts of size $n_1, n_2$, and $n_3$. Then a longest cycle in $G_3(n_1, n_2, n_3;\delta)$ has length $2(n_1+n_2) \leq 2(n-1)$.  
 \end{const}
 
 The  goal of this paper is to find a best lower bound on $\delta$ guaranteeing the existence of a $2n$-cycle in 
  a graph $G \in \cG(n,m,\delta)$ if $G$ is not only 2-connected, but $3$-connected.
 The following simple extension of Construction~\ref{con5} shows that the bound could not be larger than $ \frac{m+10}{4}$.
 
\begin{const}\label{con6}
Let $n_1 \geq n_2 \geq n_3\geq n_4\geq 1$ be such that $n_1 + n_2 + n_3 +n_4= n$.
 Let $G_4(n_1, \ldots, n_4;\delta) \in \cG(n,4\delta-9, \delta)$ be the bipartite graph obtained from $\bigcup_{j=1}^4K_{\delta-3, n_j} $  by adding $3$ vertices $a_1,a_2,a_3$, all of which are  adjacent to every vertex in the parts of size $n_1, n_2,n_3$, and $n_4$. Then a longest cycle in $G_4(n_1, \ldots, n_4;\delta)$ has length $2(n_1+n_2+n_3) \leq 2(n-1)$.  
 \end{const}
 
 \begin{figure}\label{fig:ext}
 \centering
 \includegraphics[scale=.65]{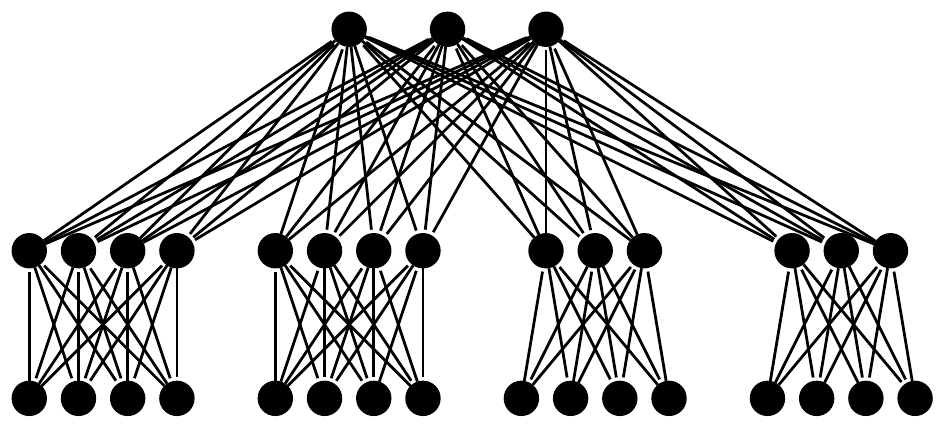}
     \caption{An example of Construction~\ref{con6}.}
 \end{figure}

The main result of the paper is that Construction~\ref{con6} is indeed extremal for $3$-connected graphs:

\begin{thm}\label{jackson6} 
Let $m,n,\delta$ be integers. If $\delta\geq \max\{n, \frac{m+10}{4}\}$,
then every $3$-connected graph $G \in \cG(n,m,\delta)$ contains a cycle of length $2n$.
\end{thm}

We discuss possible extensions of Theorem~\ref{jackson6} to $k$-connected bipartite graphs and hypergraphs in concluding remarks.
We will apply this theorem in a forthcoming paper on  so-called {\em super-pancyclic} bipartite graphs and hypergraphs. 
This notion was introduced and discussed in~\cite{KLZ}.

In the next section, we discuss how Theorem~\ref{jackson6} can be translated into the language of hamiltonian Berge cycles. 

%The setup of  bipartite graphs with large degrees in one part is  useful for finding long cycles in hypergraphs with high minimum degree.

\subsection{Hamiltonian Berge cycles in hypergraphs}

A {\em hypergraph} $\cH$ is a set of vertices $V(\cH)$ and a set of edges $E(\cH)$ such that each edge is a subset of $V(\cH)$. 
%Often we take $V(\cH) = [n]$ where $[n]:=\{1,\ldots, n\}$ is the set of the first $n$ integers.

We consider hypergraphs with  edges of any size. The {\em degree}, $d(v)$, of a vertex $v$  is the number of edges that contain $v$. 
The {\em minimum degree} of a hypergraph $\cH$ is  $\delta(\cH) := \min_{v\in V(\cH)} d(v)$. The 
{\em co-degree} of a vertex set  $A$ %, denoted $d(v_1, \ldots, v_\ell)$, 
 is the number of edges that contain  $A$.

A {\em Berge cycle} of length $\ell$ in a hypergraph is a set of $\ell$ distinct vertices $\{v_1, \ldots, v_\ell\}$ and $\ell$ distinct edges $\{e_1, \ldots, e_\ell\}$ such that $v_i, v_{i+1} \in e_i$ for every $i\in [\ell]$ (indices are taken modulo $\ell$). The vertices $\{v_1, \ldots, v_\ell\}$ are  the {\em base vertices} of the  cycle.

Naturally, a {\em Berge hamiltonian cycle} in  a hypergraph $\cH$ is a Berge cycle whose set of base vertices is $V(\cH)$.

Let $\cH = (V(\cH), E(\cH))$ be a hypergraph. The {\em incidence   graph  } of $\cH$ is the bipartite graph $I(\cH)$ with parts $(X, Y)$ where $X = V(\cH)$, $Y= E(\cH)$ such that for $e \in Y, v \in X,$ $ev\in E(I(\cH))$ if and only if the vertex $v$ is contained in the edge $e$ in $\cH$. 

If $\cH$ has $n$ vertices, $m$ edges and minimum degree at least $\delta$, then
%is a hypergraph with edge cardinalities at least $r$, then in $I(\cH)$, each vertex of $X$ has at least $r$ neighbors in $Y$. That is, 
 $I(\cH) \in \cG(n,m,\delta)$. There is a simple relation between the cycle lengths in a hypergraph $\cH$ and its incidence graph $I(\cH)$:
 If $\{v_1, \ldots, v_\ell\}$ and $\{e_1, \ldots, e_\ell\}$ form a Berge cycle of length $\ell$ in $\cH$, then $v_1 e_1 \ldots v_\ell e_\ell v_1$ is a cycle of length $2\ell$ in $I(\cH)$, and vice versa. 

%Using incidence   graphs, we also define 2-connectedness in hypergraphs. 

For a positive integer $k$, call a hypergraph {\bf $k$-connected} if its incidence   graph   is  $k$-connected.

If one would like to prove an analog of
 Dirac's theorem on hamiltonian cycles in graphs for hamiltonian Berge cycles in hypergraphs, then the bound on the minimum degree would be exponential in $n$.
 One of the examples is the following construction from~\cite{KLZ}. 
  
\begin{const}[\cite{KLZ}]\label{con4}
Let $V(\cH)=V_1 \cup V_2$ where $|V_1| =\lceil (n+2)/2 \rceil$, $|V_2| = \lfloor (n-2)/2 \rfloor$, $V_1 \cap V_2 = \emptyset$, and  let $E(\cH)=E_1\cup E_2$, where $E_1$ is the set of all subsets $A$ of $V(\cH)$ of size $\lceil n/4\rceil$ such that
$|V_1\cap A|=1$ (and $|V_2 \cap A| = \lceil n/4 \rceil - 1$), and $E_2=\{V_1\}$. 
Then $\cH$ has an exponential in $n$  minimum degree, high connectivity and positive codegree of each pair of the vertices. But 
$\cH$ has no Berge hamiltonian cycle.
%That is, for each $v_1 \in V_1 - \{v\}$ and $v_2 \in V_2 - \{v\}$, $d(v_1, v_2) = 0$.
\end{const}

On the other hand, rephrasing Theorem~\ref{jackson6} in terms of hypergraphs, we get  a reasonable and sharp bound on the minimum degree in terms of  the number of vertices and edges
that provides the existence of hamiltonian Berge cycles in $3$-connected hypergraphs.

\begin{thm}\label{mainj2} Let positive integers $n,m,\delta$ be such that 
\begin{equation}\label{nm}
\mbox{\em $ \delta \geq \max\{n, \frac{m+10}{4}\}$.}
\end{equation}
 Then every $3$-connected
$n$-vertex hypergraph with $m$ edges and minimum degree at least $\delta$   has a hamiltonian Berge cycle. \end{thm}

\subsection{Notation and outline of the proof of Theorem~\ref{jackson6} }

For a graph $G$, a cycle $C$ in $G$, and a vertex $x$ not appearing in $C$, let $t(x,C)$ denote the size of a largest $x,V(C)$-{\em fan} in $G$, i.e. the largest number  of $x,V(C)$-paths such that any two of them share only
%which are pairwise vertex-disjoint except at 
 $x$.  Since $G$ is $3$-connected, $t(x,C)\geq 3$.

%Given a cycle $C$ in $G$, a component $D$ of $G-C$ is $(D,a,b)$-{\em strong} if there is a set $A\subseteq C\cap N(D)$
%such that $|A|=a$ and for every two vertices $v,w\in A$ there is a $v,w$-path whose interior is in $D$ and contains at least $b$ vertices in $X$.

Our proof is by contradiction. We assume that for some positive integers $m,n,\delta$ with $\delta\geq \max\{n, \frac{m+10}{4}\}$,
there is a counter-example: 
 a $3$-connected $(X,Y)$-bigraph $G\in  \cG(n,m, \delta)$ with no $2n$-cycles. We study the properties of $G$.

We consider each cycle $C$ in $G$ equipped with a clockwise direction. For every vertex $u$ of $C$, $x^+_C(u)$ denotes the closest to $u$ clockwise vertex of $X$
distinct from $u$. For every vertex $u$ of $C$, $x^-_C(u)$ denotes the closest to $u$ counterclockwise vertex of $X$
distinct from $u$. For a set $U\subset V(C)$, $X_C^+(U)=\{x^+_C(u)\; : \;u\in U\}$. When $C$ is clear from the context, the subscripts could be omitted.
The vertices $ y^+(u),y^-(u)$ and the sets $X^-(U),Y^+(U),Y^-(U)$ are defined similarly.

We  consider triples $(C,x,F)$ where $C$ is a cycle, $x\in X-V(C)$ and $F$ is an $x,C$-fan. 
By $D(C,x)$ we will denote the component of $G-C$ containing $x$. By definition, $V(F)-V(C)\subseteq D(C,x)$.

\begin{definition}\label{better}
A triple $(C,x,F)$
{\bf is better than}  a triple $(C',x',F')$ if
\begin{enumerate}
\item[(a)] $|C|>|C'|$, or
\item[(b)] $|C|=|C'|$ and $t(x,C)>t(x',C')$, or
\item[(c)] $|C|=|C'|$, $t(x,C)=t(x',C')$, and $|V(F)\cap V(C)\cap Y|>|V(F')\cap V(C')\cap Y|$, or
\item[(d)] $|C|=|C'|$, $t(x,C)=t(x',C')$,  $|V(F)\cap V(C)\cap Y|=|V(F')\cap V(C')\cap Y|$,  and $|V(F)|<|V(F')|$, or
\item[(e)] $|C|=|C'|$, $t(x,C)=t(x',C')$,  $|V(F)\cap V(C)\cap Y|>|V(F')\cap V(C')\cap Y|$,  $|V(F)|=|V(F')|$ and $|V(D(C,x))|<|V(D(C',x))|$.
\end{enumerate}
\end{definition}

Choose a best triple $(C,x,F)$.
Let \begin{eqnarray*}
2\ell=|C|, &  t=t(x,C), & T=T(C,x,F)=V(F)\cap V(C),\\
& t_X=|T\cap X|,&   t_Y=|T\cap Y|.
 \end{eqnarray*}
Similarly, let $\widetilde{T}=\widetilde{T}(C,x)$ be the set of all vertices of $C$ adjacent to a vertex of $D(C,x)$, and let $\widetilde{t}=\widetilde{t}(C,x)=|\widetilde{T}|$.
By definition, $\widetilde{T}\supseteq T$ and $\widetilde{t}\geq t$.
Viewing $F$ as a tree (spider) with root $x$, any two vertices $u,v\in V(F)$ define the unique $u,v$-path $F[u,v]$ in $F$. For $u,v\in V(C)$, let $C[u,v]$ be  the clockwise $u,v$-path in $C$ and let $C^-[u,v]$ be  the counterclockwise $u,v$-path in $C$.
If $D=D(C,x)$ and $u,v\in D\cup\widetilde{T}(C,x)$, then let $P_D[u,v]$ be a longest $u,v$-path all of whose internal vertices are in~$D$.

We will analyze the properties of best triples $(C,x,F)$ and in all cases will come to a contradiction, either by finding a better triple or by proving that $m\geq  4\delta-9$.
For this, we will try to construct so called {\em good subsets} $W$ of $X\cap T$, defined later, such that total neighborhood of $W\cup \{x\}$ will be too large.
One feature of a good set will be that no two members of such set have a common neighbor outside of $C$, {\bf CON} for short. 

In the next section we prove basic properties of our best triple $(C,x,F)$. Then in Section~\ref{st3} we show that 
$t=\widetilde{t}=3$. 
Since $G$ is $3$-connected, this means that {\em for every} $x'\in X-C$, $t(x',C)=3$.
In Section~\ref{rc}, we discuss
special types of components of $G-C$ and possibilities to choose a triple $(C,x,F)$ with $x$ in such a component. 
After that we consider $T=T(C,x,F)$ and try to find 
 a $4$-element good subset of  the 
set $A=X^+(T)\cup X^-(T)$. The main obstacles will be that some members of $A$ have many common neighbors, in particular, CONs. 
Section~\ref{sec4} is devoted
to the case analysis of different types of such CONs. We conclude the paper with some comments.

\section{Preliminary lemmas}

\begin{lem}\label{tX} The following inequalities always hold:
\begin{center}(i) $\ell\geq t+t_X$; \qquad (ii) $|X|-\ell+t_X\geq 3$; \qquad
(iii) $|X|\geq t+3$.\end{center}
\end{lem}
\begin{proof}
If $w\in T\cap X$ and $y^+(w)\in T$, then the cycle $wF[w,y^+(w)]y^+(w)C[y^+(w),w]w$ is longer than $C$, a contradiction. Similarly,
$y^-(w),x^+(w),x^-(w)\notin T$. Thus, $t_X\leq \ell/2$ and $t_Y\leq \ell-2t_X$. This proves~(i).

Since $\delta\geq |X|\geq \ell+1\geq t+1= d_F(x)+1$, there is $y\in N(x)-N_F(x)$. By (d) in the definition of $(C,x,F)$, $y\notin V(F)$.
By the maximality of $t$, $y\notin V(C)-V(F)$.
Since $G$ is $3$-connected, $G-x$ has a $y,C$-fan $F'$ of size $2$. Let $x',x''$ be the neighbors of $y$ in $F'$. If, say $x'\in V(C)$, then by the maximality
of $t$, $x'\in T$. Thus $\{x,x',x''\}\subset (X-V(C))\cup (T\cap X)$. This yields (ii). Now (i) and (ii) together imply (iii).
\end{proof}

\begin{lem}\label{T+} If $w\in \widetilde{T}\cap X$, then
\begin{enumerate}
\item[(i)] $y^+(w)\notin \widetilde{T}$ and
\item[(ii)] $y^+(w)$ has no neighbors in $ X^+(\widetilde{T})-x^+(w)$.
\end{enumerate}
\end{lem}
\begin{proof} If $y^+(w)$ has a neighbor in $D=D(C,x)$, then the cycle $wP_D[w,y^+(w)]y^+(w)C[y^+(w),w]w$ is longer than $C$. This contradiction proves (i).

Suppose $y^+(w)u\in E(G)$ for some $u\in X^+(\widetilde{T})-x^+(w)$. Let $u=x^+(v)$ for $v\in \widetilde{T}-w$. Consider the cycle $C'=wC^-[w,u]uy^+(w)C[y^+(w),v]vP_D[v,w]w$. Then $C'$ is longer than $C$, unless $v\in X$ and $v$ and $w$ have a common neighbor $y$ in $D$. In the last case, $|C'|=|C|$ and the only vertex in $V(C)-V(C')$ is $y^+(v)$ which by (i) does not have neighbors in $D$. Define an $x,C'$-fan $F'$ as follows. If $y\notin V(F)$, then let $F'=F$. If $y\in V(F)$, say $y\in F[x,u_i]$ for some $u_i\in T$, then let 
$F'=F-E(F[y,u_i])$. In both cases, since $y^+(v)$  does not have neighbors in $D(C',x)\subset D$,
 the triple $(C',x,F')$ is better than $(C,x,F)$: if $y\notin V(F)$, then by (e), otherwise either by (c) or by (d).
\end{proof}

%For a triple $(G,x,F)$ let $D=D(C,x)$ be the component of $G-V(C)$ containing $x$. By definition, $F-C\subseteq D$.
\begin{lem}\label{neighborF0}
If $x_1\in X^+(\widetilde{T})$, then $x_1$ cannot have a neighbor in $D=D(C,x)$, i.e., $x_1\notin \widetilde{T}$.
\end{lem}
\begin{proof}
Suppose $x_1$ has a neighbor $y'$ in $D$. Let $u_1\in \widetilde{T}$ be such that $x_1=x^+(u_1)$ and $z$ be a neighbor of $u_1$ in $D$. Let $P$ be
a $z,y'$-path in $D$ and the cycle $C'$ be defined by  $C'=x_1C[x_1,u_1]u_1zPy'x_1$. If $y'\neq z$, then $C'$
 is longer than $C$ and we are done. Thus $z=y'$ and hence $u_1\in X$.
  In this case  $C'$ and $C$ have the same length and $t(x,C')=t(x,C)$. As in the proof of Lemma~\ref{T+}(ii), if $y'\notin V(F)$, then let $F'=F$. If $y'\in V(F)$, say $y'\in F[x,u_i]$ for some $u_i\in T$, then let 
$F'=F-E(F[y',u_i])$. In both cases, since by  Lemma~\ref{T+}(i), $y^+(u_1)$  does not have neighbors in $D(C',x)\subset D$,
 the triple $(C',x,F')$ is better than $(C,x,F)$: if $y'\notin V(F)$, then by (e), otherwise either by (c) or by (d).
\end{proof}

Given a cycle $C$ and distinct $x_1,x_2,x_3\in X\cap V(C)$, we say that {\em $x_1$ and $x_2$ cross at $x_3$} if the cyclic order is $x_1,x_3,x_2$ and 
$x_1y^+(x_3),x_2y^-(x_3)\in E(G)$ or if the cyclic order is $x_1,x_2,x_3$ and 
$x_1y^-(x_3),x_2y^+(x_3)\in E(G)$. In this case, we also say that {\em $x_3$ is crossed by $x_1$ and $x_2$}.

\bigskip
\begin{lem}\label{neighborF1}
Suppose that $x_1,x_2\in X^+( \widetilde{T})$, cross at $x_3\in X\cap V(C)$. Then $x_3\notin  \widetilde{T}$.
\end{lem}
\begin{proof}
Suppose that the cyclic order is $x_1,x_3,x_2$ and  $x_1y^+(x_3),x_2y^-(x_3)\in E(G)$ (the other case is symmetric). Let $y$ be a neighbor of $x_3$ in $D$. 
Let $u_1\in  \widetilde{T}$ be such that $x_1=x^+(u_1)$ and $z$ be a neighbor of $u_1$ in $D$. Let $P$ be
a $z,y$-path in $D$ and the cycle $C'$ be defined by 
\[
	C' := x_1y^+(x_3)C[y^+(x_3),u_1]u_1zPyx_3C^-[x_3,x_1]x_1.
\]
If $y\neq z$, then $C'$
 is longer than $C$ and we are done. Thus $z=y$. In this case,  $C'$ and $C$ have the same length and $t(x,C')=t(x,C)$. 
 As in the proof of Lemma~\ref{T+}(ii), if $y\notin V(F)$, then let $F'=F$. If $y\in V(F)$, say $y\in F[x,u_i]$ for some $u_i\in T$, then let 
$F'=F-E(F[y,u_i])$. Again as in the proof of Lemma~\ref{T+},  the triple $(C',x,F')$ is better than $(C,x,F)$.
 %However, $|F\cap C'\cap Y|>|F\cap C\cap Y|$, a contradiction to (c).
\end{proof}

Recall that for two vertices in $G$,  CON means ``a common neighbor outside of $C$."

\medskip
\begin{lem}\label{cross0} Suppose that $x_1,x_2\in X^+( \widetilde{T})$. Then
\begin{enumerate}
\item[(i)] $x_1$ and $x_2$ have no CON;
\item[(ii)] neither of $x_1$ and $x_2$ has a CON with $x$.
\end{enumerate}
\end{lem}

\begin{proof} Part (ii) follows from Lemma~\ref{neighborF0}. So, suppose $x_1$ and $x_2$ have a CON $y$, and
 $u_1,u_2\in  \widetilde{T}$ are such that $x_1=x^+(u_1)$ and $x_2=x^+(u_2)$.
 By Lemma~\ref{neighborF0}, $y\notin D$. Consider the cycle 
\[
	C' := x_1C[x_1,u_2]u_2P_D[u_2,u_1]u_1C^-[u_1,x_2]x_2yx_1.
\]
Cycle  $C'$ is longer than $C$, unless $u_1,u_2\in X$ and have a common neighbor $y'$ in $D$. In the last case, $|C'|=|C|$ and the only vertices in $V(C)-V(C')$ are $y^+(u_1)$ and $y^+(u_1)$
which by Lemma~\ref{neighborF0}(i) do not have neighbors in $D$. Define an $x,C'$-fan $F'$ as follows. If $y'\notin V(F)$, then let $F'=F$. If $y'\in V(F)$, say $y'\in F[x,u_i]$ for some $u_i\in T$, then let 
$F'=F-E(F[y',u_i])$. In both cases, since $y^+(u_1)$ and $y^+(u_2)$  do not have neighbors in $D(C',x)\subset D$,
 the triple $(C',x,F')$ is better than $(C,x,F)$: if $y\notin V(F)$, then by (e), otherwise either by (c) or by (d).
\end{proof}

\begin{lem}\label{cross1} Suppose $u_1,u_2\in \widetilde{T}$ are such that the path $P_D[u_1,u_2]$ contains an internal vertex in~$X$.
If $x_1=x^+(u_1)$ and $x_2=x^+(u_2)$ cross at $x_3\in X\cap V(C)$, then
\begin{enumerate}
\item[(i)] $x_3\notin \widetilde{T}$ and if  $x_3=x^+(u)$ where $u\in \widetilde{T}$, then $u\in Y$;
\item[(ii)] $G$ has a cycle $C'$ containing $(X\cap V(C)-x_3)\cup (X\cap P_D[u_1,u_2])$ such that  $|C'|\geq |C|$;
\item[(iii)] $x_3$ has no CON with any vertex in the set $\{x\}\cup X^+(T)$;
\item[(iv)] $x_3$ has at most $t$ neighbors on $C$.
\end{enumerate}
\end{lem}

\begin{proof}
Part (i) follows from Lemmas~\ref{T+} and~\ref{neighborF1}. The cycle 
\[
	C_1 := x_1y^+(x_3)C[y^+(x_3),u_2]u_2P_D[u_2,u_1]u_1C^-[u_1,x_2]x_2y^-(x_3)C^-[y^-(x_3),x_1]x_1
\]
proves (ii).

To prove (iii), assume that $y$ is a CON of $x_3$ with a vertex in $\{x\}\cup X^+(T)$, and
consider all cases. First note that by Lemma~\ref{neighborF1},  $y\notin D$; in particular,  $x_3$ has no CON with $x$. If $u_j\in \widetilde{T}$, $x_j=x^+(u_j)$, $yx_j\in E(G)$, and   $x_j\in C[y^+(x_3),u_1]$, then the cycle 
\[
	C' := x_1C[x_1,x_3]x_3yx_jC[x_j,u_1]u_1P_D[u_1,u_j]u_jC^-[u_j,y^+(x_3)]y^+(x_3)x_1
\]
is longer than $C$, unless
 $P_D[u_1,u_j]=u_1y'u_j$ for some $y'\in D$. If $P_D[u_1,u_j]=u_1y'u_j$,  then $|C'|=|C|$ and the only vertices in $V(C)-V(C')$ are $y^+(u_1)$ and $y^+(u_j)$
which by Lemma~\ref{neighborF0}(i) do not have neighbors in $D$. Define an $x,C'$-fan $F'$ as at the end of the proof of
Lemma~\ref{neighborF0}, and see that $(C',x,F')$ is better than $(C,x,F)$ exactly as there. Similarly, if $x_j\in C[u_1,y^-(x_3)]$, then the cycle
\[
	C' := x_2C[x_2,u_j]u_jP_D[u_j,u_2]u_2C^-[u_2,x_3]x_3yx_jC[x_j,y^-(x_3)]y^-(x_3)x_2
\]
is longer than $C$,  unless $P_D[u_j,u_2]=u_jy'u_2$ for some $y'\in D$. Again, defining $F'$ as above, we get a triple $(C',x,F')$ better than $(C,x,F)$,
 a contradiction. This proves (iii).

By the choice of $(C,x,F)$ and (ii), $x_3$ has at most $t$ neighbors on $C_1$. The only vertices in $Y\cap V(C)-V(C_1)$ are $y^-(x_1)$ and $y^-(x_2)$. If $x_3y^-(x_1)\in E(G)$, then the cycle 
\[
	y^-(x_1)C[y^-(x_1),y^-(x_3)]y^-(x_3)x_2C[x_2,u_1]u_1P_D[u_1,u_2]u_2C^-[u_2,x_3]x_3y^-(x_1)
\]
is longer than $C$. If $x_3y^-(x_2)\in E(G)$, then the cycle 
\[
	x_1C[x_1,x_3]x_3y^-(x_2)C[y^-(x_2),u_1]u_1P_D[u_1,u_2]u_2C^-[u_2,y^+(x_3)]y^+(x_3)x_1
\]
is longer than $C$. 
This proves (iv).
\end{proof}

\begin{lem}\label{cros11} Suppose $u_1,u_2\in \widetilde{T}$ are such that the path $P_D[u_1,u_2]$ contains an internal vertex in~$X$,
 $x_1=x^+(u_1)$, and $x_2=x^+(u_2)$. Then  at most one vertex in $C$ is crossed by $x_1$ and $x_2$.
\end{lem}
\begin{proof}

Suppose vertices $x_3, x_4 \in V(C) \cap X$ are crossed by $x_1$ and $x_2$. We will show first that $x_3$ and $x_4$ have no CON. Suppose there is some $y \in (N(x_3) \cap N(x_4)) - V(C)$. By Lemma~\ref{cross1}, $y \notin V(D)$.

We consider two cases. If $x_3$ and $x_4$  both are on $C[x_1,x_2]$ or both are on $C[x_2, x_1]$, then we may assume that their  cyclic order is $x_1, x_3, x_4, x_2$. In this case, the cycle 
\[
	x_1 C[x_1, x_3] x_3 y x_4 C[x_4, u_2] u_2 P_D[u_2, u_1] u_1C^-[u_1, x_2] x_2y^-(x_4) C^-[y^-(x_4), y^+(x_3)] y^+(x_3)x_1
\]
(see Figure~\ref{fig:crossCON}, left) is longer than $C$.

If one of $x_3$ and $x_4$ is  on $C[x_1, x_2]$ and the other is on $C[x_2, x_1]$, then we may assume that their cyclic order is $x_1, x_3, x_2, x_4$. 
In this case, the cycle
\[
	x_1 C[x_1, x_3] x_3 y x_4 C^-[x_4, x_2] x_2 y^+(x_4) C[y^+(x_4), u_1] u_1 P_D[u_1, u_2] u_2C^-[u_2, y^+(x_3)] y^+(x_3) x_1
\]
(see Figure~\ref{fig:crossCON}, right) is longer than $C$.
This proves that $x_3$ and $x_4$ have no CON. 

\begin{figure}
\centering
\includegraphics[scale=.5]{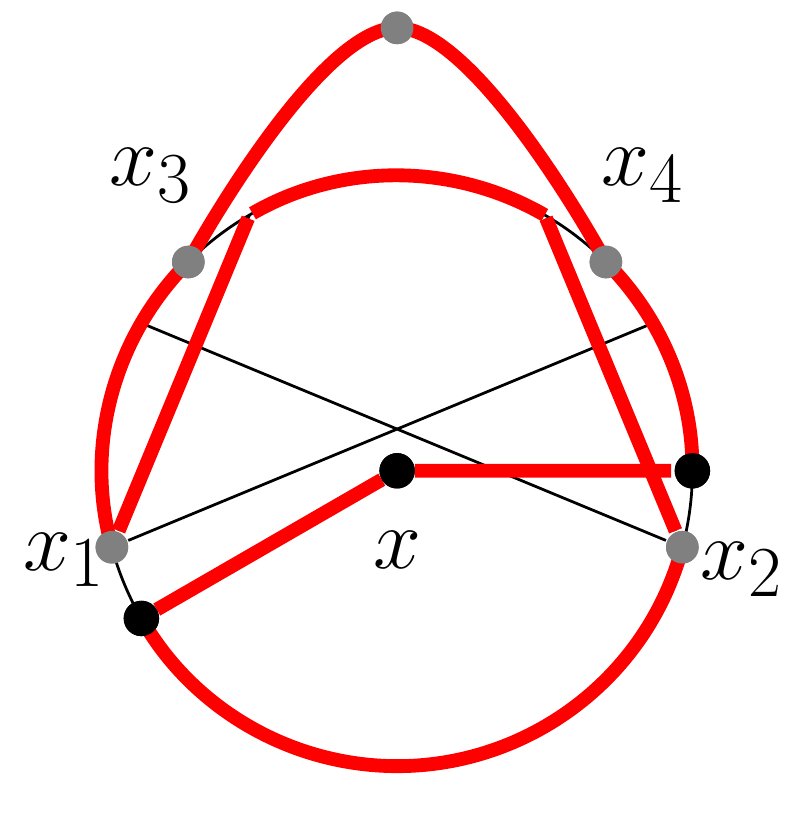}
\includegraphics[scale=.5]{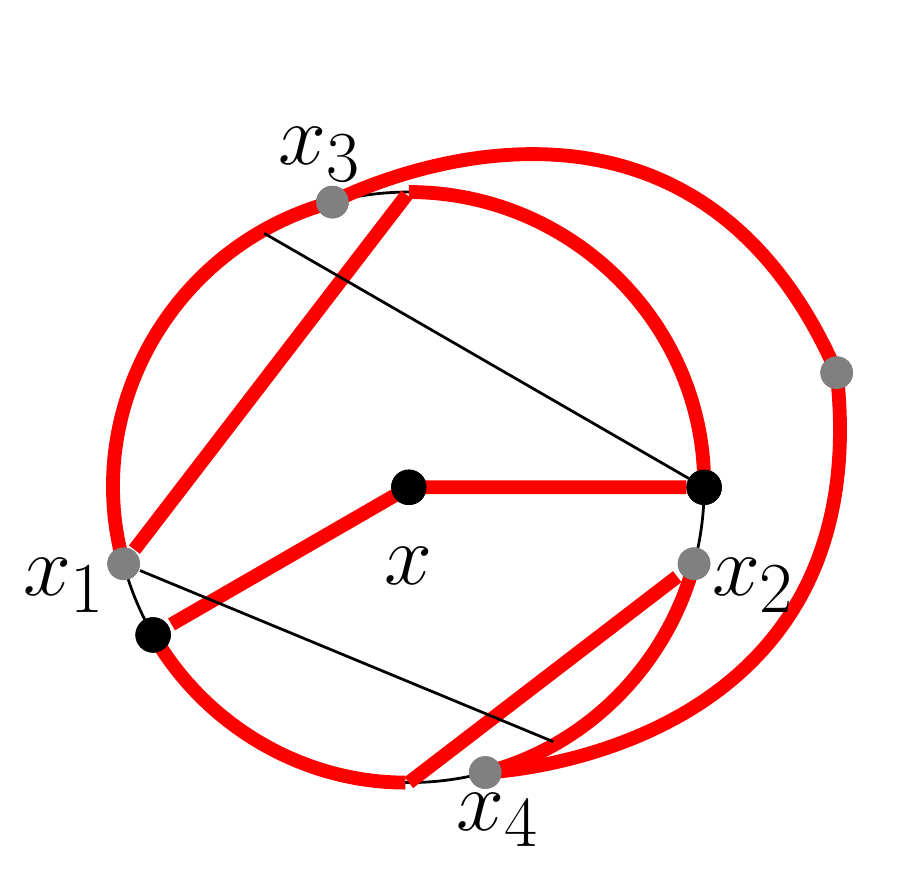}
\caption{Longer cycles when $x_1$ and $x_2$ in Lemma~\ref{cros11} have multiple crossings.}
\label{fig:crossCON}
\end{figure}

 Let $A=X^+(T)\cup\{x,x_3,x_4\}$
(possibly, $X^+(T)\cap\{x_3,x_4\}\neq \emptyset$), and $A'=A-\{x,x_3,x_4\}$. Note that $|A'|\geq t-2$. 

By definition, $|N(x)-C|\geq \delta-t_Y$. By Lemma~\ref{cross1}(iv),  $|N(x_3)-C|\geq \delta-t$ and $|N(x_4)-C|\geq \delta-t$. By Lemma~\ref{T+},
\begin{equation}\label{sum}
\sum_{u\in A'} |N(u)\cap V(C)|\leq \ell|A'|-t_X|A'|+\min\{t_X,|A'|\}.
\end{equation}
% each $u\in X^+(T)-\{x_0,x_0'\}$. 
% $|N(u)-V(C)|\geq \delta-\ell$ for
 By Lemmas~\ref{cross0} and~\ref{cross1}(iii), no two distinct vertices in $A$ have a CON.
Thus, using~\eqref{sum} and remembering about the $\ell$ vertices in $Y\cap V(C)$, we get \allowdisplaybreaks
\begin{align*}
|Y|	&\geq \ell+\sum_{u\in A}|N(u)-V(C)|\\
	&= \ell+|N(x)-V(C)|+|N(x_3)-V(C)|+|N(x_4)-V(C)|+\sum_{u\in A'}|N(u)-V(C)|\\
	&\geq \ell+(\delta-t_Y)+(\delta-t)+(\delta-t)+(\delta|A'|-\sum_{u\in A'} |N(u)\cap V(C)|)\\
	&\geq \ell+(|A'|+3)\delta-2t-t_Y-(\ell-t_X)|A'|-\min\{t_X,|A'|\}\\
	&\geq \ell+(t-2+3)\delta-2t-(t-t_X)-(\ell-t_X)(t-2)-\min\{t_X,t-2\}\\
	&= \ell+(t+1)\delta-(3t-t_X)-(\ell-t_X)(t-2)-\min\{t_X,t-2\} \\
	&\geq \ell+(t+1)\delta-3t-(\ell-t_X)(t-3+1) \\
	&=(t+1)\delta-3t-(\ell-t_X)(t-3)+t_X.
\end{align*}
Since by Lemma~\ref{tX}, $\delta\geq \ell-t_X+3$, this yields
\begin{align*}
|Y|	&\geq (t+1)\delta-3t-\delta(t-3)+3(t-3)+t_X =4\delta-9+t_X.
\end{align*}
This contradiction proves the lemma.
\end{proof}

The following lemma holds for any bipartite graph $G$ (no restrictions on minimum degree or connectivity).

\begin{lem}\label{cros12}
Let $C$ be a cycle of $G$, and let $u, v\in V(C) \cap X$. If $u$ and $v$ have at most $a$ crossings, then $d_C(u) + d_C(v) \leq |V(C)|/2 + 2 + a$.\end{lem}

\begin{proof}
We induct on $a$. Suppose $a = 0$. Consider the two paths $P_1=C[u,v]$ and $P_2=C^-[u,v]$. In $P_1=v_1\ldots v_k$ ($v_1 = u, v_k = v)$, each $v_i \in X$ satisfies at most one of the following: $v_{i+1}u \in E(G)$ or $v_{i-1}v \in E(G)$. So $d_{P_1}(u) + d_{P_1}(v) \leq |V(P_1) \cap X|$. Similarly, $d_{P_2}(u) + d_{P_2}(v) \leq |V(P_2) \cap X|$. Since $(X \cap V(P_1)) \cap (X \cap V(P_2)) = \{u,v\}$ and $V(P_1) \cup V(P_2) = V(C)$, we get $d_C(u) + d_C(v) \leq |V(C)|/2 + 2$. 

For $a \geq 1$, delete an edge incident to $u$ that is used in a crossing, and apply induction.
\end{proof}

%\begin{lem}\label{cros12}
%For $u,v  \in C \cap X$,
%\begin{enumerate}
%\item[(i)] If $u$ and $v$ have no crossings, then $d_C(u)+d_C(v)\leq \ell+2$. More generally, if $u$ and $v$ have $a$ crossings, then $d_C(u)+d_C(v)\leq \ell+2+a$.
%\item[(ii)]If $u$ and $v$ have no two consecutive crossings, then $d_C(u)+d_C(v)\leq 1.5\ell+2$.
%\end{enumerate}
%\end{lem}
%
%\begin{proof} Our Lemma 2.5 in the submitted paper is Part (i) for $a=0$.
%{\bf Need to elaborate!}
% For general $a$, the induction step is a deleting a crossing edge.{\bf Need to elaborate!}
%
%For Part (ii), if   $d_C(u)+d_C(v)> 1.5\ell+2$, then  do the following: Suppose we have a crossing at $w$. If $w\in C[u,v]$, delete the crossing edge incident to $u$, and if
%$w\in C[v,u]$, delete the crossing edge incident to $v$.
%Repeat this until there is no crossings.
%By (i), the total number of the remaining edges incident with $u$ and $v$ is at most $\ell+2$. So, if we deleted more than $\ell/2$ edges,
%the ends of some two of them are consecutive $Y$-vertices on $C$. If these two vertices are in the same component of $C-u-v$, then they define two
%consecutive crossings, and we are done. Otherwise, these two vertices are $y^-(u)$ and $y^+(u)$ or $y^-(v)$ and $y^+(v)$. Suppose the former. But 
%since $y^+(u)\in  C[u,v]$,
%$uy^+(u)$ must be a crossing edge, which it is not.
%\end{proof}

\section{Bounds on $t$ and $\widetilde{t}$ in best triples}\label{st3}

Recall that  $(C,x,F)$ is a best triple, $D=D(C,x)$ is the component of $G-V(C)$ containing $x$,  $T = V(F) \cap V(C)$, and $\widetilde{T}= N_C(D)$.

A set of vertices $W=\{x_1, \ldots, x_k\} \subseteq X \cap V(C)$ is {\bf good} if 
\begin{enumerate}
\item[(i)] $d_C(x) \leq k$,
\item[(ii)] the vertices of $\{x\} \cup W$ pairwise have no CON, and
%\item if $u_i$ is matched to $u_j$, then $G[D \cup \{u_i, u_j\}]$ contains a path $P_D[u_i, u_j]$ with at least one internal vertex in $X$, and
\item[(iii)] we can partition $W$ into sets $W_1,\ldots,W_s$ such that for each $j\in [s]$, $|W_j|\geq 2$ and any two distinct
vertices in $W_j$ cross at no more than one vertex in $C$. 
\end{enumerate}

\begin{lem}\label{good} If $W$ is a good set, then $|W| < \max\{4,t\}$. \end{lem}
\begin{proof} Suppose $k\geq \max\{4,t\}$ and $W=\{x_1, \ldots, x_k\}$
 is a good set. Note that $\delta \geq |X| \geq |W| \geq 4$. Let $(W_1,\ldots,W_s)$ be a partition of $W$ satisfying (iii)
 in the definition of a good set.
 By Lemma~\ref{cros12}, if $x_i$ and $x_j$ have at most one crossing, then $d_C(x_i) + d_C(x_j) \leq \ell + 3$. Hence 
\[
	\sum_{i=1}^k d_C(x_i) =\sum_{j=1}^s\sum_{w\in W_j}d_C(w)\leq k(\ell + 3)/2.
\]
Since $|Y\cap V(C)|=\ell$, $\delta(G) \geq \delta$ and $k\geq t$, we get
\[
	|Y|\geq \ell+(k+1)\delta-t-k\frac{\ell+3}{2}\geq \ell\left(1-\frac{k}{2}\right)+k\left(\delta-\frac{5}{2}\right)+\delta.
\]
Since the net coefficient at $\ell$ is negative and $\ell\leq |X|-1\leq \delta-1$, this is at least
$k\left(\frac{\delta}{2}-2\right)+2\delta-1$. Now the net coefficient at $k$ is nonnegative, so the minimum is attained at
$k=4$. Hence $|Y|\geq 
4\delta-9$, a contradiction.
\end{proof}

Next, we show that both $t$ and $\widetilde{t}$ are small.

\begin{lem}\label{t3}  $t = 3$.
\end{lem}

\begin{proof}Since $G$ is 3-connected, $t= |T| \geq 3$. Suppose $t\geq 4$. We claim that  $X^+(T)$ is a good set. 

Since 
$F$ is a largest $x, C$-fan,
 $x$ has at most $t$ neighbors in $C$. By Lemma~\ref{cros11}, for any $x_i, x_j \in X^+(T)$, $x_i$ and $x_j$ have at most one crossing in $C$.  By  Lemma~\ref{cross0},  no two distinct vertices in $X^+(T) \cup \{x\}$ have a CON.  This certifies that  $X^+(T)$  is good, a contradiction to Lemma~\ref{good}.
%Thus,  since $|Y\cap C|=\ell$, we get
%$$|Y|\geq \ell+(t+1)\delta-(t-t_X)-\frac{\ell+2}{2}t.
%$$
%Since the net coefficient at $\ell$ is $1-t/2\leq -1$ and $\ell\leq \delta-1$, this is at least
%$\frac{t+4}{2}\delta-1+t_X-2t$. Now the coefficient at $t$ is $\frac{\delta-4}{2}>0$. So, it is enough to check $t=4$, and we have
%$$|Y|\geq 4\delta-1+t_X-8\geq 4\delta-9,$$
%a contradiction.
\end{proof}

\begin{lem}\label{td3} $| \widetilde{T}|=3$.
\end{lem}
\begin{proof}
We have $T \subseteq \widetilde{T}$. %So we will prove that $| \widetilde{T}|\leq 3$.
Suppose $| \widetilde{T}|\geq 4$. Choose a set $U=\{u_1,\ldots,u_4\}\subseteq  \widetilde{T}$ so that
 $T=\{u_1,u_2,u_3\}$, and $u_4\in  \widetilde{T}-T$.
  Let $P$ be a shortest path from $u_4$ to $F-C$ in $G[D+u_4]$.
Let $j\in [3]$ be such  that the end, $p$, of $P$ distinct from $u_4$ belongs to the $x,u_j$-path in $F$. Assume 
$[3]=\{j,j',j''\}$.
The path $u_{j'}F[u_{j'}
, p] pPu_4$ contains an internal vertex in $X$ (namely, $x$). Partition $U$ into $U'=\{u_4,u_{j'}\}$ and $U''=\{u_j,u_{j''}\}$.

By  Lemma~\ref{cros11}, each of the pairs $U'$ and $U''$ has at most one crossing in $C$. 
Since 
$F$ is a largest $x, C$-fan,
 $x$ has at most $t$ neighbors in $C$.
%The vertex $x$ has at most $t$ neighbors in $C$, otherwise we could find a larger fan. 
 By  Lemma~\ref{cross0},  no two distinct vertices in $X^+(U) \cup \{x\}$ have a CON. This certifies that $X^+(U)$ is good, a contradiction to Lemma~\ref{good}.
%   Hence by  Lemma~\ref{cros12}
%  $d_C(x^+(u_j))+d_C(x^+(u_{j'}))\leq \ell+3$ and $d_C(x^+(u_4))+d_C(x^+(u_{j''}))\leq \ell+3$.
% By~\eqref{f81}, no two distinct vertices in $U\cup \{x\}$ have a CON.
% Hence
%$$|Y|\geq \ell+5\delta- 2(\ell+3)-3=5\delta-\ell-9\geq 4\delta-8,$$
%a contradiction.
%So $U$ is a good set. 
\end{proof}
%\qed
%\section{$|D\cap C|=3$}

%\bigskip

\begin{remark}Lemma~\ref{td3} implies that $T = \widetilde{T}$, i.e.,  the only vertices in $C$ with neighbors in $D$ are the vertices of $T$. In particular, no vertex in $V(C) - T$ has a CON with $x$. 

\end{remark}
\subsection{More structure and fewer crossings}\label{rc}
One of the results of this section is that for any best triple $(C, x, F)$, no vertices in $X^+(T)$ cross in $C$. 
Recall that by Lemma~\ref{t3}, $|T| = |V(F) \cap V(C)| =3$. 
%We want to find a component $D$  of $G-C$ such that there are no crossings for the vertices in $X^+(T)$ where $T$ is the set of end vertices of
%an $(x,C)$-fan with $3$ paths for some $x\in D\cap X$.

A component $D$ of $ V(G)-C$ is {\bf 2-rich} if there is a set $U=\{u_1,u_2,u_3\}= V(C)\cap N(D)$ such that for all distinct $i,j$,
$D$ contains a $u_i,u_j$-path with at least two internal vertices in $X$.

%So, if possible, then we choose a $2$-rich component. If there are no $2$-rich components, then we stick to our original rules (a)--(e).

\begin{lem}\label{t3r} If $|T\cap X|\leq 1$,
 then $D$ is $2$-rich.
\end{lem}
\begin{proof}Suppose  $T=\{u_1,u_2,u_3\}$ where $u_1,u_2\in Y$. 
If some $y\in D\cap Y$ is not adjacent to $u_3$, then all $y,C$-paths contain internal vertices in $X$, and hence $D$ is $2$-rich. Thus we may assume that each $y\in D\cap Y$ is adjacent to $u_3$. In particular, $u_3\in X$.

By Rule (d) of  Definition~\ref{better}, $d_F(x)=t=3$, so because $\delta\geq |X|+1\geq t+3+1\geq 7$, there is
 $y'\in N(x)$ with $y'x\notin E(F)$. Since $G$ is $3$-connected, it contains a
$y',C$-fan  $F'$ with $3$ paths. Recall that $y'u_3$ is one of such paths.
For $i=1,2$, let $P_i$ be the $y',u_i$-path in $F'$ and $v_iy'\in E(P_i)$. Suppose that for $i=1,2$, there is $y_i\in N(v_i)-C-y'-P_{3-i}$ (possibly, $y_2=y_1$).
Then $D$ is $2$-rich: $P_1\cup P_2$ connects $u_1$ with $u_2$, and for $i\in \{1,2\}$, path $u_3y_iv_iy'P_{3-i}$ connects $u_3$ with $u_{3-i}$; and each of these three paths contains $\{v_1,v_2\}\subset X$. Hence by symmetry we may assume that every neighbor of $v_1$ is in $V(C)\cup P_2$. Note $N(v_1)\cap V(C)\subseteq\{u_1,u_2,u_3\}$, since $| \widetilde{T}|=3$. Then the cycle
$v_1y'P_2C[u_2,u_1]v_1$ has at least $2\delta$ vertices, a contradiction.
\end{proof}

\begin{lem}\label{cros001} Suppose $D$ is not 2-rich. For any $x' \in X \cap V(C)$, $G-x'$ has no cycle $C'$ such that
\begin{enumerate}
\item[(i)] $X\cap V(C')\supseteq X\cap V(C)-x'+x$, and
\item[(ii)] $C'$ contains the neighbors $y^+(x')$ and $y^-(x')$ of $x'$ on $C$.
\end{enumerate}
\end{lem}
\begin{proof}
Suppose we have  $C'$ satisfying~(i) and~(ii). If we have strict containment in~(i), then $|C'|>|C|$, contradicting (a) in the choice of $(C, x, F)$.
Thus  $X\cap V(C')= X\cap V(C)-x'+x$. 

Let $D'$ be the component of $G - V(C')$ containing $x'$. Let $M$ be the set of neighbors of $D'$ on $C'$.
By (ii), $\{y^+(x'),y^-(x')\}\subset M$. Since $G$ is $3$-connected, $D'-\{y^+(x'),y^-(x')\}$
contains an $x',C'$-path $P$. Then $P$ together with the edges $x' y^+(x')$ and $x'y^-(x')$
forms an $x',C'$-fan $F'$ with $|V(F') \cap V(C') \cap Y| \geq 2$. Moreover since $D$ was not 2-rich, by Lemma~\ref{t3r}, $|V(F) \cap V(C) \cap Y| \leq 1$. So $(C', x', F')$ is a better triple than $(C, x, F)$, a contradiction.
\end{proof}

\begin{lem}\label{nocrossings}No two vertices in $X^+(T)$ cross in $C$.\end{lem}
\begin{proof}Suppose $x_i = x^+(u_i)$ and $x_j=x^+(u_j)$ cross at some vertex $x_0 \in V(C) \cap X$. By symmetry, we may assume that their  cyclic order is $x_i, x_0, x_j$. Let 
\[
	C' := x_i C[x_i y^-(x_0)] y^-(x_0) x_j C[x_j, u_i] u_i P_D[u_i, u_j] u_j C^-[u_j, y^+(x_0)] y^+(x_0) x_i.
\] 
If $D$ is 2-rich, then $P_D[u_i, u_j]$ has at least 2 internal vertices in $X$, and so $C'$ is longer than $C$. If $D$ is not 2-rich, then $C'$ satisfies conditions (i) and (ii) of Lemma~\ref{cros001}, a contradiction.
\end{proof}

Let $e$ be an edge of $C$, let $u,v \in V(C)$, and let $P$ be any $u,v$-path containing $e$, which we orient from $u$ to $v$. We say that $P$ and $C$ \emph{agree on} the edge $e$ if the orientation of $e$ (oriented from $u$ to $v$) in the $u,v$-segment of $C$ containing $e$ is the same as the orientation of $e$ in $P$. 

\begin{lem}\label{path-to-crossings}
Let $u,v \in X \cap V(C)$. Suppose that there is a $u,v$-path $P$ with $(X \cap V(C)) \cup \{x\}\subseteq V(P)$ and there exists some $z, z' \in V(P)$ such that $V(P) \cap V(D) = V(P[z, z'])$, i.e., $P$ enters and leaves $D$ exactly once. Then
\begin{enumerate}
\item[(i)] $u$ and $v$ have no common neighbor outside of $P$, and 
\item[(ii)] if $P$ and $C$ agree on an edge $e$, then $u$ and $v$ cannot have a crossing at an endpoint of $e$.
\end{enumerate}

\end{lem}

\begin{proof}
Note that $x \in P[z,z']$.  If $u$ and $v$ had a common neighbor outside $P$, then we could extend $P$ to a cycle longer than $C$, so (i) holds.

To prove (ii), suppose that $P$ and $C$ agree on an edge $e$ which lies on $C[u,v]$, $w \in X \cap V(C)$ is an endpoint of $e$, and $u$ and $v$ cross at $w$. Suppose that the edges of $C[u,v]$ incident to $w$ are $y'w$ and $wy''$, so that $uy''$ and $vy'$ are the two edges forming $u$ and $v$'s crossing on $w$. Without loss of generality, $e = y'w$. The condition that $P$ and $C$ agree on $e$ guarantees that $P[u,w]$ contains $y'$.

There are two cases to consider: either both $y'w$ and $wy''$ are edges of $P$, or just $y'w$. 

%In the first case, form a cycle $C'$ by joining $P[u,y']$, edge $y'v$, $P[v,y'']$, and edge $y''u$. This cycle $C'$ contradicts Lemma~\ref{cros001}: $C'$ contains $x$ and all vertices of $X \cap V(C)$ except for $w$, and also contains $y'$ and $y''$: the neighbors of $w$ on $C$.

In the first case, let $C' := uP[u, y'] y'v P[v, y''] y'' u$. Then $V(C') \supseteq V(C)  - \{w\}+\{x\}$. If we have strict containment, then $|C'| > |C|$, a contradiction. So we may assume $V(C') \cap X = (V(C) \cap X) -\{w\} + \{x\}$. Observe that $C'$ satisfies Lemma~\ref{cros001} for $x' = w$. So $D$ is 2-rich. Let $a$ be the vertex in $P$ preceding $z$ and $a'$ the vertex in $P$ succeeding $z'$ (so $a, a' \in V(C)$). Let $P'$ be a $a,a'$-path internally disjoint from $C$ that contains at least 2 internal vertices in $X$. Let $C''$ be obtained by replacing in $C'$ the segment $P[a, a']$ with $P'$. We have  $|V(C'') \cap X| > |X + \{x\} - \{w\}|$. Therefore $|C''| > |C|$, a contradiction.

In the second case, the cycle $uP[u,y']y'vP[v,w]wy''u$ is longer than $C$, since it contains all of $X \cap V(C)$ as well as $x$,  a contradiction.
\end{proof}

\medskip
%And I think now we took care of crossings.

\section{Handling the case $\widetilde{t}=3$}\label{sec4}

\subsection{Short, medium, and long-type configurations}\label{types}

We continue to study properties of a best triple $(C,x, F)$. 
Recall that %Lemma~\ref{good} states that there are no good sets. B
 by Lemma~\ref{td3},  $\widetilde{t}=|\widetilde{T}| = 3$, so  we will assume that 
%
%\begin{lem}\label{good}
%If $X'\subset X$ is good, $|X'|\geq 4$ and $x$ has no common neighbors outside $V(C)$ with any $x'\in X$, then $|Y|\geq 4\delta-7$.
%\end{lem}
%\begin{proof}
%In this case, we have 
%\begin{align*}
%|Y|&\geq(|X'|+1)\delta-(\ell+2)\frac{|X'|}{2}-|X'|+\ell\\
%&=(|X'|+1)\delta-\ell(\frac{|X'|}{2}-1)-2|X'|\\
%&\geq(|X'|+1)\delta-(\delta-1)(\frac{|X'|}{2}-1)-2|X'|\\
%&=\delta(\frac{|X'|}{2}+2)-1-\frac{3|X'|}{2}\\
%&\geq 2\delta+4(\frac{\delta}{2}-\frac{3}{2})-1\\
%&=4\delta-7. \qedhere
%\end{align*}
%\end{proof}
  $N(D)\cap V(C)=\{u_{1},u_{2},u_{3}\}$. Partition $V(C)- \{u_1,u_2,u_3\}$ into $U_1$, $U_2$ and $U_3$, where for $i\in [3]$, 
$U_i=V(C[u_i,u_{i+1}])- \{u_i, u_{i+1}\}$, i.e.\ $U_i$ is the set of vertices on $C$ from $u_{i}$ to $u_{i+1}$ not including either endpoint. Here and in the remainder of the paper, we let the indices on $D$'s neighbors wrap around modulo $3$, so that, for example, $u_{0} = u_{3}$ and $u_{4} = u_{1}$.

Let $X_i = U_i \cap X$ and $Y_i = U_i \cap Y$. For $j > 0$, let $x_{i,j}$ be the $j$\textsuperscript{th} vertex in $X_i$ clockwise; let $x_{i,-j}$ be the $j$\textsuperscript{th} vertex in $X_{i-1}$ counterclockwise. For example, $x_{i,1} = x^+(u_i)$ and $x_{i,-1} = x^-(u_i)$. Define $y_{i,j}$ similarly.

One of the lines of attack in this section is trying to find a $4$-element good subset of $X^+(T) \cup X^-(T)$, which will contradict  Lemma~\ref{good}. 
This will not work if several of these vertices have many CONs. We will classify the obstacles to this approach into three types. For each $i \in [3]$, we say that:
\begin{itemize}
\item $i$ has \emph{short type} if $x_{i,-1}$ and $x_{i,1}$ have a CON.
\item $i$ has \emph{medium type} if $x_{i,1}$ and $x_{i+1,-1}$ have a CON.
\item $i$ has \emph{long type} if $x_{i,-1}$ and $x_{i+1,1}$ have a CON.
\end{itemize}
These three configurations are shown in Figure~\ref{fig:sml}.

\begin{figure}
    \centering
    \includegraphics[scale=.55]{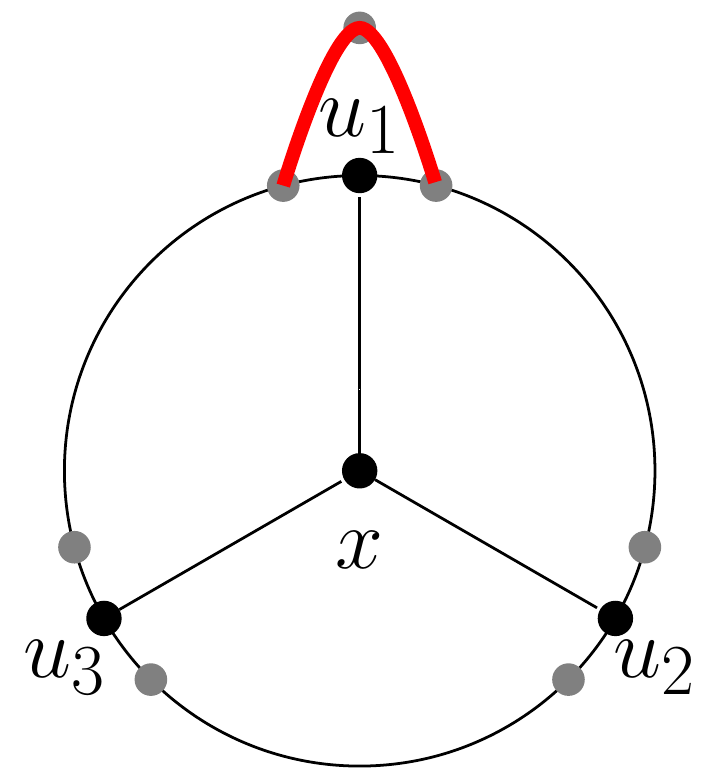}
    \includegraphics[scale=.55]{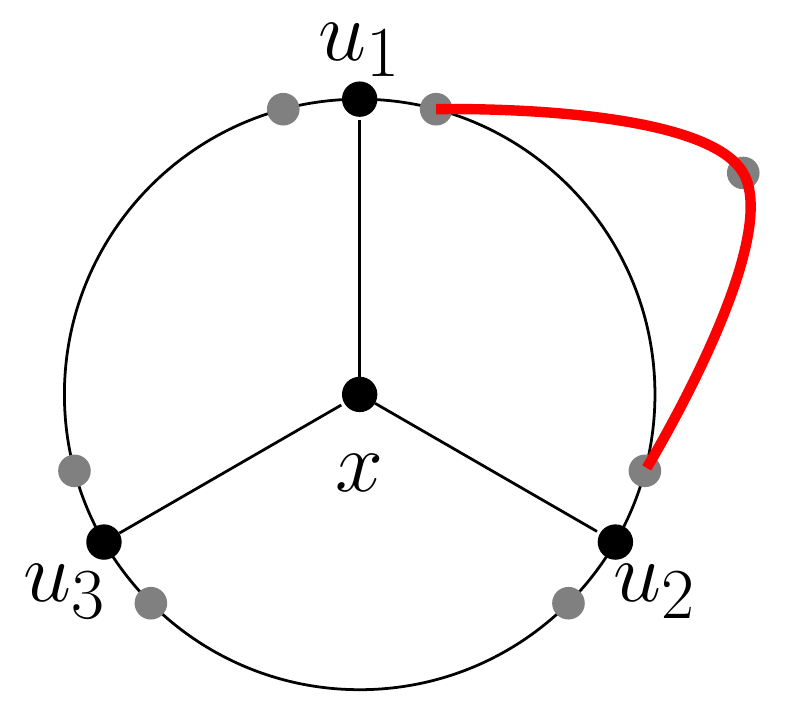}
    \includegraphics[scale=.55]{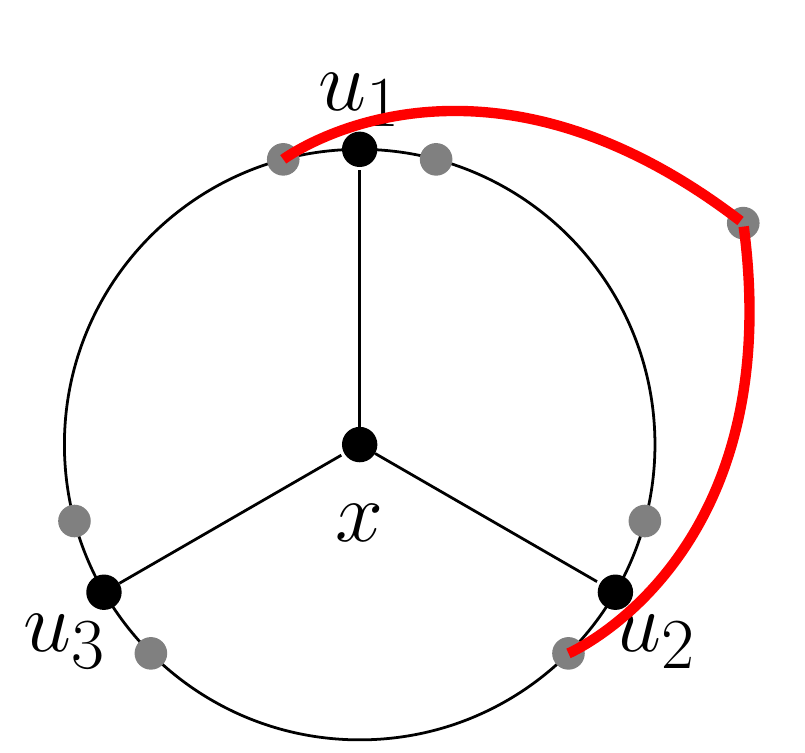}
    \caption{Short-type, medium-type, and long-type configurations.}
    \label{fig:sml}
\end{figure}

%They are the only possible CONs between these six vertices, since by Lemma~\ref{cross0}, $x_{i,1}$ and $x_{i+1,1}$ cannot have a CON for any $i$, nor can $x_{i,-1}$ and $x_{i+1,-1}$. From the same lemma, it follows that no vertex $x_{i, \pm 1}$ can have a CON with $x$.
We first prove that each segment $U_i$ contains at least two vertices in $X$.

\begin{lem}\label{sum8}  For any $x' \in X^+(T)$, $d_C(x) + d_C(x') \geq 8$. In particular, $d_C(x') \geq 5$.
\end{lem}

\begin{proof} %Without loss of generality, s
Suppose $x'=x_{1,1}$ and $d_C(x) + d_C(x') \leq 7$. No two vertices in
the set $X^+(T) \cup \{x\}$  have a CON or cross in $C$. By Lemmas~\ref{nocrossings} and~\ref{cros12}, $d_C(x_{2,1}) + d_C(x_{3,1}) \leq \ell + 2$. Therefore 
\[
	|Y| \geq 4 \delta - (d_C(x) + d_C(x_{1,1})) - (d_C(x_{2,1}) + d_C(x_{3,1})) + \ell
		\geq 4\delta - 7 - (\ell + 2) + \ell = 4\delta - 9.
\]
This  contradiction proves $d_C(x) + d_C(x') \geq 8$. Since $d_C(x) \le t = 3$, $d_C(x') \geq 5$.
\end{proof}

\begin{lem}\label{shortsegment}
For each $i \in [3]$, $x_{i,1} \neq x_{i+1, -1}$.\end{lem}
\begin{proof}
%
%Without loss of generality, suppose $x_{1,1} = x_{2,-1}$. Let $C' = u_1 F[u_1, u_2] C[u_2, u_1] u_1$. Then $|C'| \geq |C|$, since $X\cap V(C') \subseteq X \cap V(C) - \{x_{1,1}\} + \{x\}$. So we may assume equality holds. We claim that $d_C(x_{1,1}) \leq 3$. 
%
%Note that $N_C(x_{1,1}) - N_{C'}(x_{1,1}) \subseteq \{y^-(x_{1,1}), y^+(x_{1,1})\}$ where $y^-(x_{1,1})$ belongs to this set if and only if $u_1 \in X$. Similarly, $y^+(x_{1,1})$ belongs to it if and only if $u_2 \in X$. Let $t_X' = |\{u_1, u_2\} \cap X|$.
%
%Let $F'$ be any  $x_{1,1},C'$-fan of size 3. Then if $y \in V(C') \cap N(x_{1,1})$, the edge $x_{1,1}y$ is used in $F$. By the choice of $(C, x, F)$ as a best triple, $|V(F) \cap V(C) \cap Y| \geq |V(F') \cap V(C') \cap Y| \geq d_{C'}(x_{1,1})$. 
%Also, $|V(F) \cap V(C) \cap Y| \leq 3 - t_X'$. So we get $d_C(x_{1,1}) = d_{C'}(x_{1,1}) + t_X' \leq 3.$
%
%Let $A = \{x, x_{1,1}, x_{2,1}, x_{3,1}\}$. By Lemma~\ref{cross0}, no two vertices in $A$ have a CON. Moreover, by Lemma~\ref{nocrossings}, $x_{2,1}$ and $x_{3,1}$ have no crossings. Therefore by Lemma~\ref{cros12},
%$|Y| \geq 4 \delta - (\ell + 2) - 3 - 3 = 4\delta - 8,$ a contradiction.

Let $C' = u_i F[u_i, u_{i+1}] u_{i+1}C[u_{i+1}, u_i] u_i$. Then $|C'| \geq |C|$. If the component $D$ of $G - C$ containing $x$ is 2-rich, then $|C'| > |C|$. So by Lemma~\ref{t3r}, $|T \cap Y| \leq 1$, and hence $d_C(x) \leq 1$. By the choice of $(C,x,F)$ as a best triple,  $d_{C'}(x_{i,1}) \leq 1$ as well.  Since $V(C) -\{y^-(x_{i,1}), y^+(x_{i,1}),x_{i,1}\} \subseteq V(C')$,  $N_C(x_{i,1}) \subseteq N_{C'}(x_{i,1}) \cup \{y^-(x_{i,1}), y^+(x_{i,1})\}$, and therefore $d_{C}(x_{i,1}) \leq 1 + 2$. This contradicts Lemma~\ref{sum8}.
\end{proof}

It is possible that some segments $U_i$ contain only two vertices of $X$, but in that case, we can deduce some additional structure we will use later.

\begin{lem}\label{shortsegment2}
For each $i \in [3]$, if $x_{i,2} = x_{i+1,-1}$, then $i+1$ does not have short or long type.
\end{lem}
\begin{proof}
If $i+1$ has short or long type, we can find a cycle $C'$ such that $X \cap C'$ includes $x$ but leaves out $x_{i,1}$.

If $i+1$ has short type and $y$ is a CON of $x_{i+1,-1}$ and $x_{i+1,1}$, then
\[
	C' := u_{i+1} C^-[u_{i+1}, x_{i+1,-1}] x_{i+1,-1} y x_{i+1,1} C[x_{i+1,1}, u_i] u_i F[u_i, u_{i+1}] u_{i+1}.
\]
Note that $C$ includes at most three vertices of $Y$ which are not in $C'$: $y^+(x_{i,1})$, possibly $y^-(x_{i,1})$ (if $u_i \in X$), and possibly $y^+(u_{i+1})$ (if $u_{i+1} \in X$).

If $i+1$ has long type and $y$ is a CON of $x_{i+1,-1}$ and $x_{i-1,1}$, then
\[
	C' := x_{i+1,-1} C[x_{i+1,-1}, u_{i-1}] u_{i-1}F[u_{i-1}, u_i]u_i C^-[u_i, x_{i-1,1}]x_{i-1,1} y x_{i+1,-1}.
\]
Again, $C$ includes at most three vertices of $Y$ which are not in $C'$: $y^+(x_{i,1})$, possibly $y^-(x_{i,1})$ (if $u_i \in X$), and possibly $y^+(u_{i-1})$ (if $u_{i-1} \in X$).

In both cases, $|C'| \ge |C|$, with strict inequality if $D$ is 2-rich. So we may assume $D$ is not 2-rich. By Lemma~\ref{t3r}, $|T \cap Y| \leq 1$, and hence $d_C(x) \leq 1$. Therefore by the choice of $(C, x, F)$ as a best triple, $d_{C'}(x_{i,1}) \leq 1$. Then in either case $d_{C}(x_{i,1}) \leq 1 + 3$, contradicting Lemma~\ref{sum8}.
%
%If $x_{i,1}$ has two neighbors $y_1, y_2 \in V(C) \cap V(C')$, then let $F'$ be the $x_{i,1}, C'$-fan consisting of the edges $x_{i,1}y_1, x_{i,1}y_2$ and the paths $C[x_{i,1}, x_{i,2}], C^-[x_{i,1}, u_i]$. The triple $(C', x_{i,1}, F')$ would be better than $(C, x, F)$ by (b) in  Definition~\ref{better}, a contradiction. 
%
%So $x_{i,1}$ has at most one neighbor in $V(C) \cap V(C')$. In both cases above, $V(C) - V(C')$ includes at most three vertices in $Y$. Even if all three are neighbors of $x_{i,1}$, we still have $d_C(x_{i,1}) \le 4$.
%
%Let $A = \{x, x_{1,1}, x_{2,1}, x_{3,1}\}$. By Lemma~\ref{cross0}, no two vertices in $A$ have a CON. Moreover, by Lemma~\ref{nocrossings}, $x_{i-1,1}$ and $x_{i+1,1}$ have no crossings. Therefore by Lemma~\ref{cros12},
%\[
%	|Y| \ge \ell + \sum_{u \in A} |N(u) - V(C)| \ge \ell + 4 \delta - (\ell + 2) - 3 - 4 = 4\delta-9,
%\]
%a contradiction.
\end{proof}

\begin{lem}\label{type-lemma}
For each $i \in [3]$, one of the following  configurations must appear:
\begin{enumerate}
\item[(i)] $i$ has short type, or
\item[(ii)] one of $i-1$ or $i$ has medium type, or
\item[(iii)] $i+1$ has long type.
\end{enumerate}
\end{lem}
\begin{proof} Suppose for some $i\in [3]$ none of (i)--(iii) holds.
Let $W = \{x_{i-1,1}, x_{i,-1}, x_{i,1}, x_{i+1,-1}\}$. By Lemma~\ref{nocrossings} (applied to $C$ and also to the backward orientation of $C$), the vertices inside the sets $W_1=\{x_{i-1, 1}, x_{i,1}\}$ and $W_2=\{x_{i,-1}, x_{i+1,-1}\}$ have no crossings. By Lemma~\ref{cross0}, no vertex in $W$ can have a CON with $x$. Since by Lemma~\ref{good},  $W$ is not a good set, some
 two vertices in $W$  have a CON. %, or else $W$ will be a good set, contradicting Lemma~\ref{good}.
By Lemma~\ref{cross0} again, $x_{i-1,1}$ and $x_{i,1}$ have no CONs, and $x_{i,-1}$ and $x_{i+1,-1}$ have no CONs. This leaves the configurations in the statement of this lemma.
\end{proof}

The plan of the  remainder of this paper is as follows:
%In the remainder of this paper, we deal with the different configurations in several subsections, each of which ends by obtaining a contradiction:
\begin{enumerate}

\item In the next subsection we define abundant indices and show that not all $i\in [3]$ are abundant. This will help to handle medium-type and short-type configurations.

\item In Subsection~\ref{two-long-section} we show  that at most one $i \in [3]$ has long type.

\item In Subsection~\ref{medium-section} we prove that no $i \in [3]$ has medium type.
An important part of this proof is Lemma~\ref{manycon} from Subsection~\ref{many-cn-section}.
%considers the case that some $i \in [3]$ has medium type. Once this is ruled out, Lemma~\ref{type-lemma} implies that each $i \in [3]$ has either short type or long type.

\item In Subsection~\ref{one-long-section} we show  that none  of $i \in [3]$ has long type.
So, by Lemma~\ref{type-lemma}, every $i\in [3]$ has short type.

\item Subsection~\ref{short-section} finishes the proof of the main theorem by
handling the case that every $i \in [3]$ has short type.
\end{enumerate}

\subsection{On abundant indices}
\label{many-cn-section}

Call an  $i \in [3]$ {\em abundant} if each of the vertices $x_{i,2}, x_{i,3}, \dots, x_{i+1, -2}$  has a CON with $x_{i,1}$ and a CON with $x_{i+1,-1}$. 

\begin{lem}\label{manycon}
At least one $i \in [3]$ is not abundant.
\end{lem}

\begin{proof} Suppose all $i \in [3]$ are abundant.
%Partition $V(C)$ into $3$ intervals $U_i$, where $U_i= V(C[y^+(w_i),w_{i+1}])-w_i$ for $i\in [3]$.
For $i\in [3]$, let %$X_i=X\cap U_i$, $Y_i=Y\cap U_i$ and
 $w_i=y^+(x^-(u_i))$. In other words,
$w_i= u_i$ if $u_i \in Y$, and $w_i = y_{i,-1}$ if $u_i \in X$. Define $W= \{w_1,y^-(w_1), w_2, y^-(w_2),w_3,y^-(w_3)\}$.
 We claim that for all $i \in [3]$,
\begin{equation}\label{fewneighbors}
N_C(x_{i,1}) \subseteq Y_i \cup W.
\end{equation}

Suppose that $x_{i,1}$ has a neighbor $y_{j,k}$ where $j \neq i$ and $y_{j,k} \in Y_j-\{ w_{j+1},y^-(w_{j+1})\}$. 
By Lemma~\ref{T+},  if $u_{j} \in X$, then $y_{j,k} \neq y_{j,1}$. So $y_{j,k}$ lies strictly between $x_{j,1}$ and $x_{j+1, -2}$. 
Since $j$ is abundant,   $x^+(y_{j,k})$ and $x_{j,1}$ have a CON, say  $y$.
Then the cycle
\[
	C' := x_{i,1} C[x_{i,1}, u_j] u_jF[u_j, u_i]u_i C^-[u_i, x^+(y_{j,k})]x^+(y_{j,k})  y x_{j,1}C[x_{j,1}, y_{j,k}] y_{j,k} x_{i_1}
\]
(see Figure~\ref{fig:manyCON}) is longer than $C$,
a contradiction. This proves~\eqref{fewneighbors}.

\begin{figure}
\centering
\includegraphics[scale=.5]{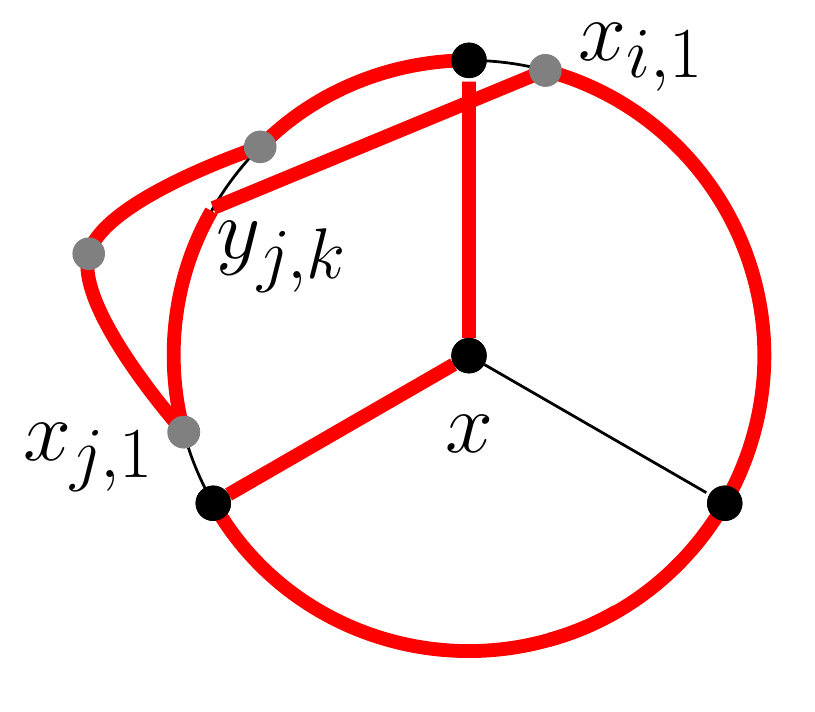}
\caption{A longer cycle when $x_{i,1}$ has a neighbor $y_{j,k}$.}

\label{fig:manyCON}
\end{figure}
Next we show that 
\begin{equation}\label{ij2}
\mbox{\em if $|X_j|=2$ and $x_{i,1}y^+(x_{j,1})\in E(G)$, then
 $N(x_{j,1})\cap  W
=\{y^+(x_{j,1})\}$.
}\end{equation}
Indeed, let $P_1$ be a longest $u_j,u_i$-path all internal vertices of which are in $D=D(C,x)$. Consider the cycle
\[
	C'' := x_{i,1} C[x_{i,1}, u_j]u_j P_1u_i C^-[u_i, y^+(x_{j,1})]y^+(x_{j,1})
 x_{i_1}.
\]
 If $D$ is $2$-rich, then $C''$ is longer than $C$, a contradiction. Thus $D$ is not $2$-rich, and hence by
 Lemma~\ref{t3r}, $|Y\cap T|\leq 1$. In this case, $|C''|\geq |C|$. Let $F''$ be a best $x_{j,1},C''$-fan. Since the triple 
 $(C'',x_{j,1},F'')$ is
 not better than $(C,x,F)$, $|C''|=|C|$ and $|N(x_{j,1})\cap V(C'')|\leq 1$. Since $y^+(x_{j,1})\in N(x_{j,1})$ by definition,
 and $W\subseteq V(C'')$,~\eqref{ij2} follows.

Now we show that similarly to~\eqref{ij2},
\begin{equation}\label{ij3}
\mbox{\em if $|X_j|\geq 3$ and $x_{i,1}y^-(w_{j+1})\in E(G)$, then
 $|N(x_{j,1})\cap  W|\leq 1$.
}\end{equation}
Indeed, let $P_1$ be a longest $u_j,u_i$-path all internal vertices of which are in $D=D(C,x)$. Since 
$|X_j|\geq 3$ and $j$ is abundant,   $x_{j+1,-1}$ and $x_{j,2}$ have a CON, say  $y$.
Consider the cycle
\[
	C''' := x_{i,1} C[x_{i,1}, u_j]u_j P_1u_i C^-[u_i, x_{j+1,-1}]x_{j+1,-1}yx_{j,2}C[x_{j,2}, y^-(w_{j+1})]y^-(w_{j+1})
 x_{i_1}.
\]
 If $D$ is $2$-rich, then $C'''$ is longer than $C$, a contradiction. Thus $D$ is not $2$-rich, and  by
 Lemma~\ref{t3r}, $|Y\cap T|\leq 1$. In this case, $|C''|\geq |C|$. Let $F'''$ be a best $x_{j,1},C'''$-fan. Since the triple 
 $(C''',x_{j,1},F''')$ is
 not better than $(C,x,F)$, $|C'''|=|C|$ and $|N(x_{j,1})\cap V(C''')|\leq 1$. Since
  $W\subseteq V(C''')$,~\eqref{ij3} follows.

If there are no distinct $i,j\in [3]$ such that $x_{i,1}y^-(w_{j+1})\in E(G)$, then by~\eqref{fewneighbors},
$\sum_{i\in [3]} |N_C(x_{i,1})|\leq \sum_{i\in [3]} (|Y_i|+2)$, and hence
\begin{equation}\label{l219}
\sum_{i\in [3]} N_C(x_{i,1})\leq \ell+6.
\end{equation}
If there is only one $j\in [3]$ such that $y^-(w_{j+1})$ is adjacent to $x_{j-1,1}$ or  to $x_{j+1,1}$
(say, $x_{i,1}y^-(w_{j+1})\in E(G)$), then  by~\eqref{fewneighbors},
$|N_C(x_{i,1})|\leq |Y_i|+3$ for $i\neq j$, but by~\eqref{ij2} and~\eqref{ij3}, $|N_C(x_{j,1})|\leq |Y_j|$. 
So again~\eqref{l219} holds.

Finally, if there are distinct $j_1,j_2\in [3]$ such that $x_{i_s,1}y^-(w_{j_s+1})\in E(G)$ for  $s\in [2]$ and some $i_s$,
then by~\eqref{ij2} and~\eqref{ij3}, $|N_C(x_{j_s,1})|\leq |Y_{j_s}|$, and  by~\eqref{fewneighbors},
$|N_C(x_{i,1})|\leq |Y_i|+4$ for $i\in [3]-\{j_1,j_2\}$. Thus~\eqref{l219} holds in all cases.

By Lemma~\ref{cross0}, no two vertices in the set $A = \{x, x_{1,1}, x_{2,1}, x_{3,1}\}$ have a CON.  Therefore,
by~\eqref{l219},
$|Y| \geq \ell + 4 \delta - 3 - (\ell - 6) = 4\delta - 9,$
a contradiction.
\end{proof}

\subsection{Eliminating multiple long-type configurations}
\label{two-long-section}

\begin{lem}\label{onelong}
At most one $i\in[3]$ has long type.
\end{lem}
\begin{proof} Suppose the lemma does not hold. By symmetry,
 we may assume that  $x_{3,-1}$ and $x_{1,1}$ have a CON $a$, and $x_{1,-1}$ and $x_{2,1}$ have a CON $b$. Since $x_{1,1}$ and $x_{2,1}$ cannot have a CON,  $a\neq b$. Consider the cycle \[
 	C' := u_3C[u_3,x_{1,-1}]x_{1,-1}bx_{2,1}C[x_{2,1},x_{3,-1}]x_{3,-1}ax_{1,1}C[x_{1,1},u_2]u_2F[u_2,u_3]u_3
\]
formed as shown in Figure~\ref{fig:2long}.

\begin{figure}[h!]
    \centering
    \includegraphics[scale=.5]{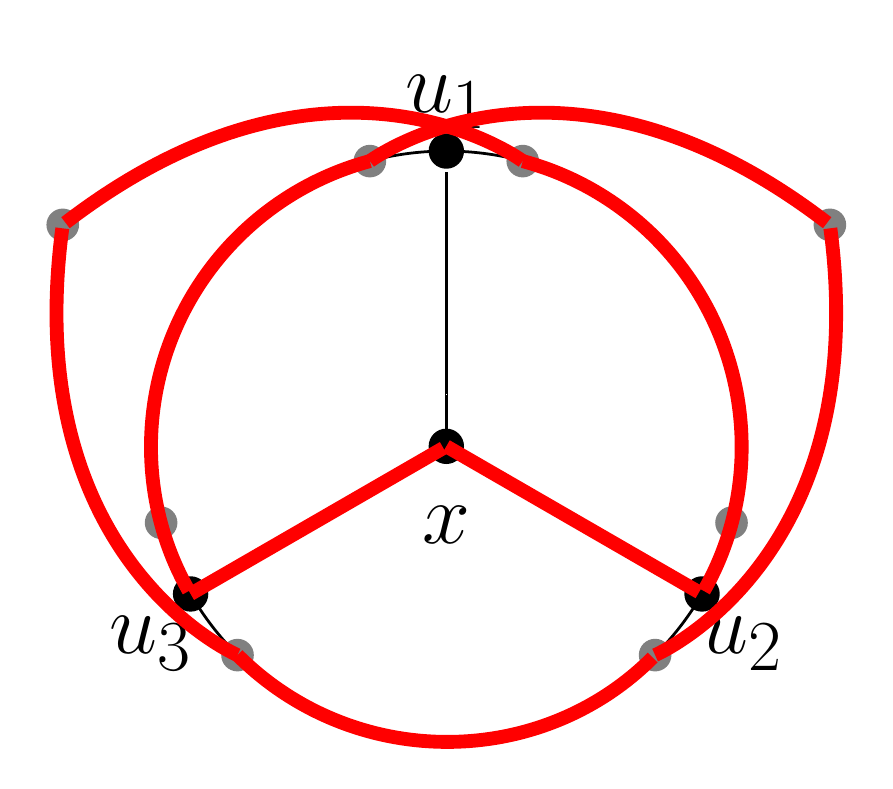}
    \caption{The cycle $C'$ formed by two long-type configurations.}
    \label{fig:2long}
\end{figure}

Cycle $C'$ includes $x$ and all vertices of $X \cap V(C)$, except possibly $u_1$, hence $|C'| \geq |C|$. If $u_1 \in Y$, $C'$ is longer than $C$, which is a contradiction. 
%$C'$ is also longer than $C$ if the $x,u_2$-path or the $x,u_3$-path of $F$ have any vertices of $X$ in their interior, so we may assume that this is not the case.
Moreover, if $F[u_2,u_3]$ contains at least 2 internal $X$ vertices, then $|C'| > |C|$. 

If $u_1 \in X$, let $yu_1$ be the last edge of the $x,u_1$-path of $F$. As $G$ is 3-connected, there is a path $P$ from $y$ to $V(C) \cap V(C')$ not containing $x$ or $u_1$. Since by definition, deleting $\{u_1, u_2, u_3\}$ disconnects $x$, and therefore $y$, from $C$, path $P$ must go from $y$ to some vertex $u'$ on either the $x, u_2$-path or the $x, u_3$-path in $F$. Without loss of generality, assume $u'$ is on the $x, u_2$-path.

Consider the cycle 
\[
	C'' := u_2C^-[u_2,x_{1,1}]x_{1,1}ax_{3,-1}C^-[x_{3,-1},x_{2,1}]x_{2,1}bx_{1,-1}C^-[x_{1,-1},u_3]u_3F[u_3,y]yPu'F[u',u_2]u_2
\]
shown in Figure~\ref{fig:2long-uy}, obtained from $C'$ by replacing the segment $C'[u', x]$ contained in $F$ by the union of $P$ and $F[x,y]$. This is longer than $C'$ (and therefore longer than $C$) except in one case: when each of $P$ and the  $F[x,y]$ is a single edge, and $u' =u_2$ (which must then be  in $X$).  In this  case,
\[
	C'' := u_2C^-[u_2,x_{1,1}]x_{1,1}ax_{3,-1}C^-[x_{3,-1},x_{2,1}]x_{2,1}bx_{1,-1}C^-[x_{1,-1},u_3]u_3F[u_3,x]xyu_2.
\]

\begin{figure}[h!]
    \centering
    \includegraphics[scale=.5]{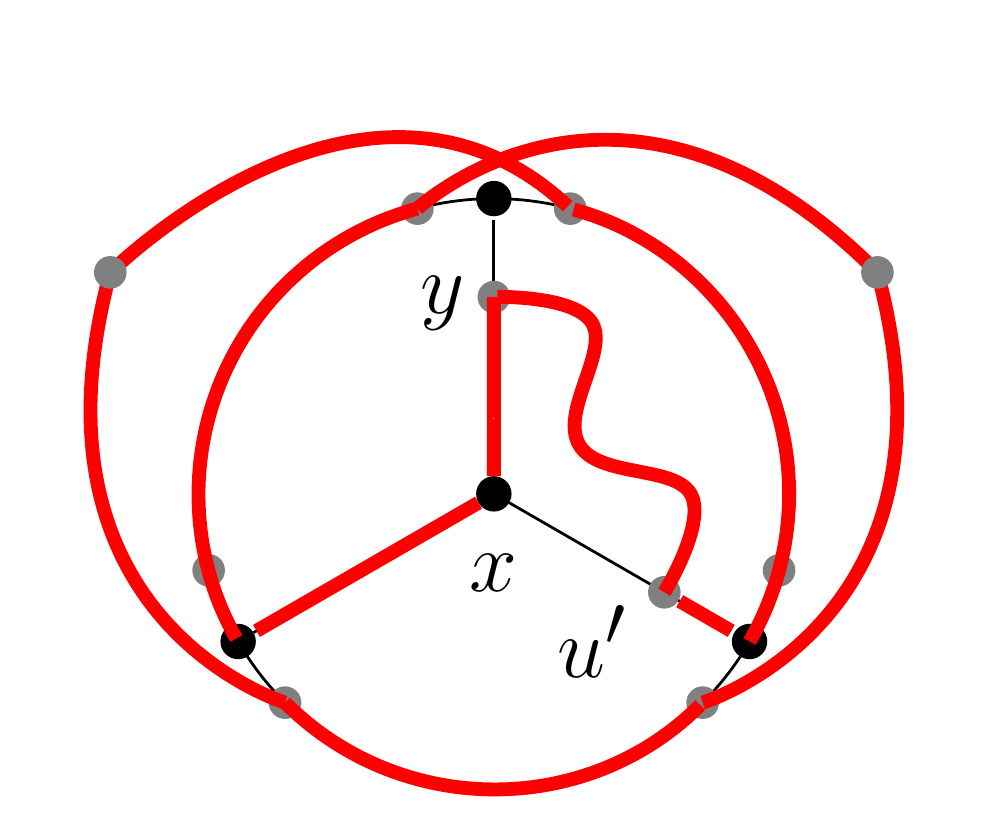}
    \caption{The cycle $C''$ formed using the path $P$.}
    \label{fig:2long-uy}
\end{figure}

Let $F''$ be the $u_1, C''$-fan formed by the paths $C[x_{1,-1}, u_1]$ and $C[u_1, x_{1,1}]$, and the edge $u_1y$. The triple $(C'', u_1, F'')$ has $|C''| = |C|$ and $t(u_1, C'') = t(x, C)$, so by our choice of the triple $(C, x, F)$, we must have 
$|V(F'') \cap V(C'') \cap Y| \le |V(F) \cap V(C) \cap Y|$. Since $V(F'') \cap V(C'') \cap Y = \{y\}$, $|V(F) \cap V(C) \cap Y| \ge 1$, which can only happen if $u_3 \in Y$. Therefore the $x,u_3$-path in $F$ consists of a single edge $xu_3$, and the only vertices of $V(C'')-V(C)$ are $x$, $y$, $a$, and $b$.

Let $y'$ be the vertex of $F$ between $x$ and $u_2$ on the $x,u_2$-path of $F$. Since $G$ is 3-connected, there is a path $P'$ from $y'$ to $V(C) \cup V(C'')$ not containing $x$ or $u_2$. However, we know that deleting $\{u_1, u_2, u_3\}$ disconnects $x$, and therefore $y'$, from $C$. Therefore either $P'$ goes from $y'$ to a vertex in $V(C'') - V(C)$, which can only be $y$, or else $P'$ goes from $y'$ to one of the vertices $u_1, u_3$.

\begin{figure}[h!] 
    \centering
    \includegraphics[scale=.5]{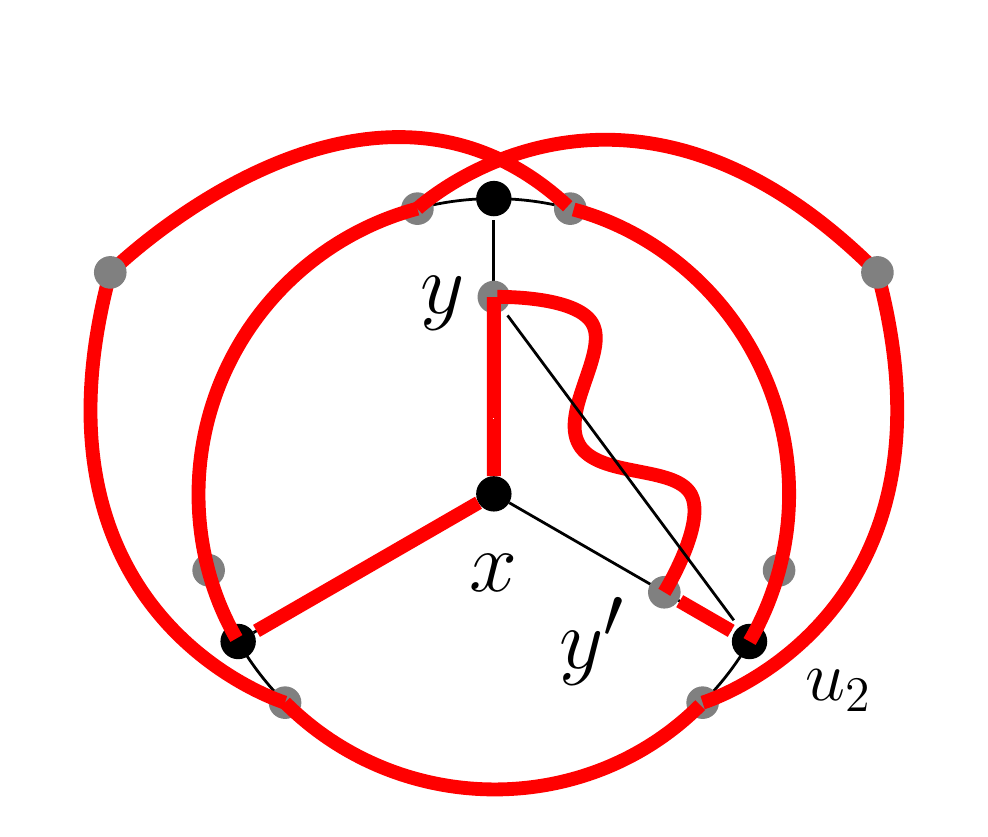}
    \includegraphics[scale=.5]{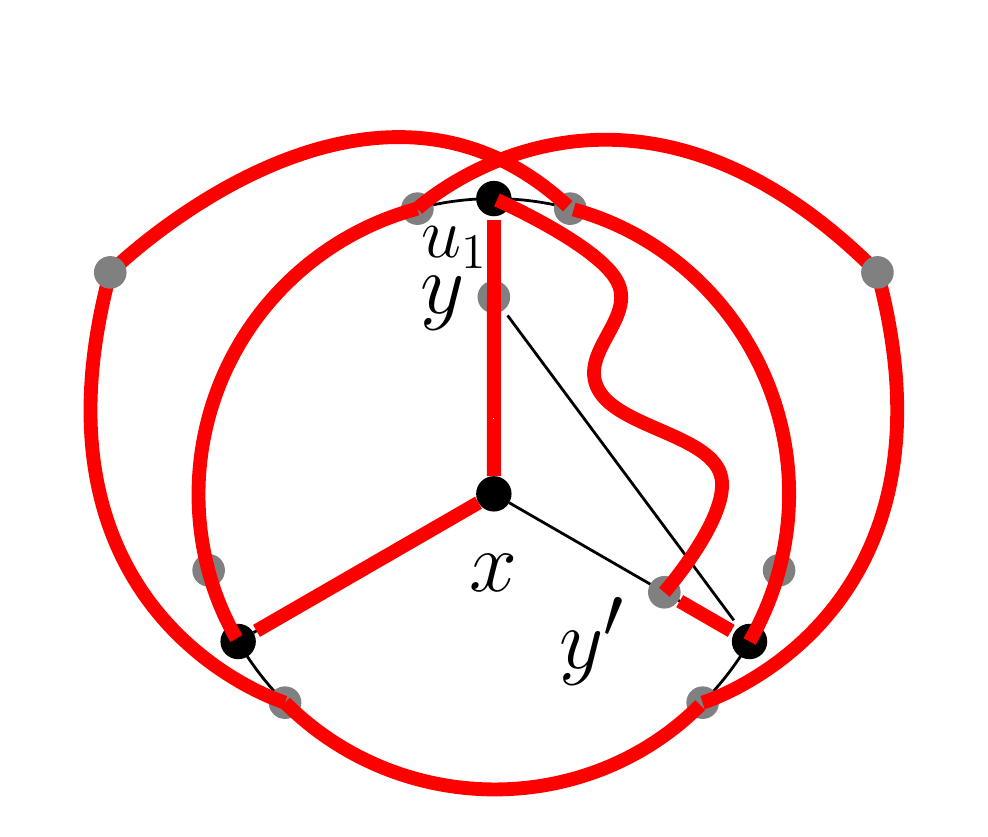}
    \includegraphics[scale=.5]{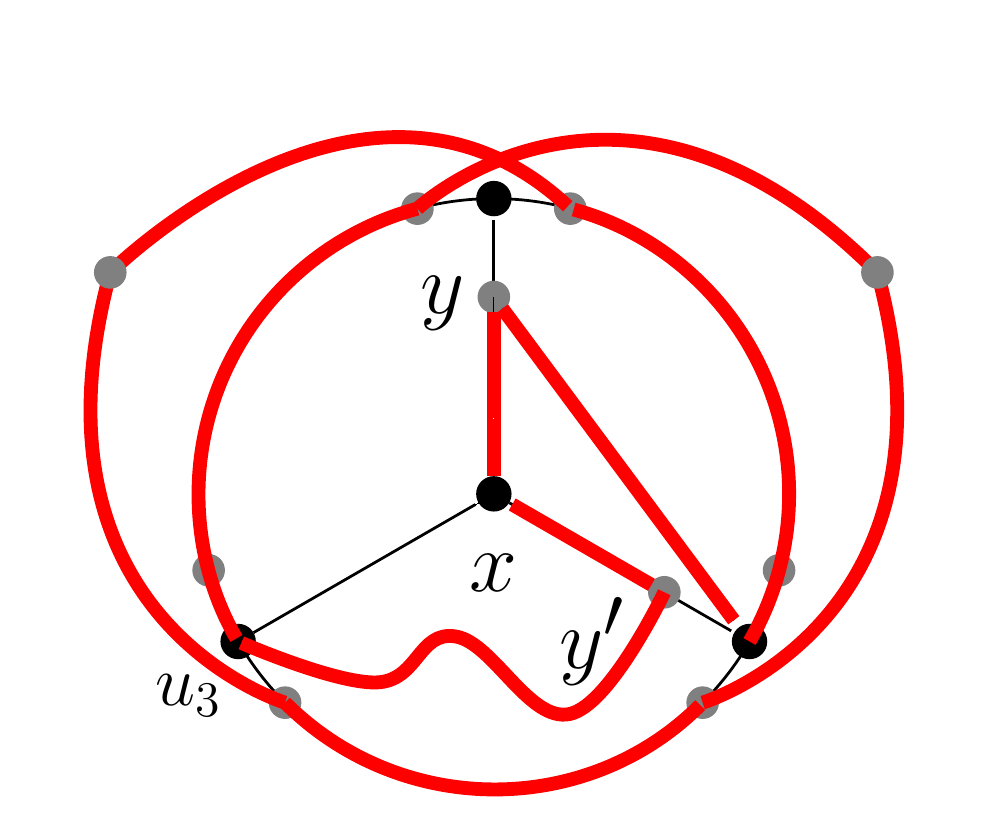}
    \caption{Three ways to extend $C''$ to a longer cycle}
    \label{fig:yprime}
\end{figure}

In each of these cases, we obtain a longer cycle. If $P'$ goes from $y'$ to $y$, we can extend $C''$ by replacing  edge $u_2y$ with $u_2y'$ followed by $P'$ to get the cycle 
\[
	u_2C^-[u_2,x_{1,1}]x_{1,1}ax_{3,-1}C^-[x_{3,-1},x_{2,1}]x_{2,1}bx_{1,-1}C^-[x_{1,-1},u_3]u_3F[u_3,x]xyP'y'u_2,
\]
as shown on the left in Figure~\ref{fig:yprime}. If $P'$ goes from $y'$ to $u_1$, we can extend $C''$ by replacing  edge $u_2y$ with $u_2y'$, $P'$, and $u_1y$ to get the cycle 
\[
	u_2C^-[u_2,x_{1,1}]x_{1,1}ax_{3,-1}C^-[x_{3,-1},x_{2,1}]x_{2,1}bx_{1,-1}C^-[x_{1,-1},u_3]u_3F[u_3,u_1]u_1P'y'u_2,
\]
as shown in the middle of Figure~\ref{fig:yprime}. Finally, if $P'$ goes from $y'$ to $u_3$, we can extend $C''$ by replacing edge $xu_3$ with $xy'$ followed by $P'$ to get the cycle
\[
	u_2C^-[u_2,x_{1,1}]x_{1,1}ax_{3,-1}C^-[x_{3,-1},x_{2,1}]x_{2,1}bx_{1,-1}C^-[x_{1,-1},u_3]u_3P'y'F[y',y]yu_2,
\]
as shown on the right in Figure~\ref{fig:yprime}.
\end{proof}

Thus, no more than one $i\in[3]$ can have long type.

\subsection{Eliminating medium-type configurations}
\label{medium-section}

In this subsection, our goal is to show that no $i \in [3]$ has medium type.

Recall that $i \in [3]$ is {\em abundant} if each of the vertices $x_{i,2}, x_{i,3}, \dots, x_{i+1, -2}$  has a CON with $x_{i,1}$ and a CON with $x_{i+1,-1}$. 

\begin{lem}\label{lots}
If $i \in [3]$ has medium type, then $i$ is abundant.
\end{lem}
\begin{proof}
Without loss of generality, we will assume that $i=1$ has medium type. We will show that for all $j\geq 1$, $x_{2,-1}$ and $x_{2,-j}$ share a CON. This is the same as showing $x_{2,-1}$ and $x_{1,a}$ share a CON for all $a \ge 1$ such that $x_{1,a} \ne x_{2,-1}$. Showing that $x_{1,1}$ and $x_{1,j}$ have a CON is symmetric.

Suppose there is an $a$ such that $x_{1,a}$ shares no CON with $x_{2,-1}$, but $x_{1,a'}$ does for all $1\leq a'<a$. Our goal is to show $\{x_{1,-1},x_{2,-1},x_{3,-1},x_{1,a}\}$ is a good set. Let $y'$ be the common neighbor of $x_{2,-1}$ and $x_{1,a-1}$. Note that $x_{1,-1},x_{2,-1},x_{3,-1}$ can have no CON by Lemma~\ref{cross0}. Additionally, by Lemma~\ref{nocrossings}, $x_{2,-1},x_{3,-1}$ have no crossings. 

By our choice of $a$, vertices $x_{1,a}$ and $x_{2,-1}$ have no CON. By Lemma~\ref{path-to-crossings} via the path
\[
	P := x_{1,a}C[x_{1,a},x_{2,-1}]x_{2,-1}y'x_{1,a-1}C^-[x_{1,a-1},u_1]u_1F[u_1,u_2]u_2C[u_2,x_{1,-1}]x_{1,-1}
\]
shown in Figure~\ref{fig:5_1}, $x_{1,a}$ and $x_{1,-1}$ have no crossings and no CON outside $P$. However, $y'$ is the only possible CON of $x_{1,a}$ and $x_{1,-1}$ on $P$, and if $x_{1,-1} y' \in E(G)$, $x_{1,-1}$ and $x_{2,-1}$ would have a CON, which also is impossible.
\begin{figure}[h!]
    \centering
    \includegraphics[scale=.5]{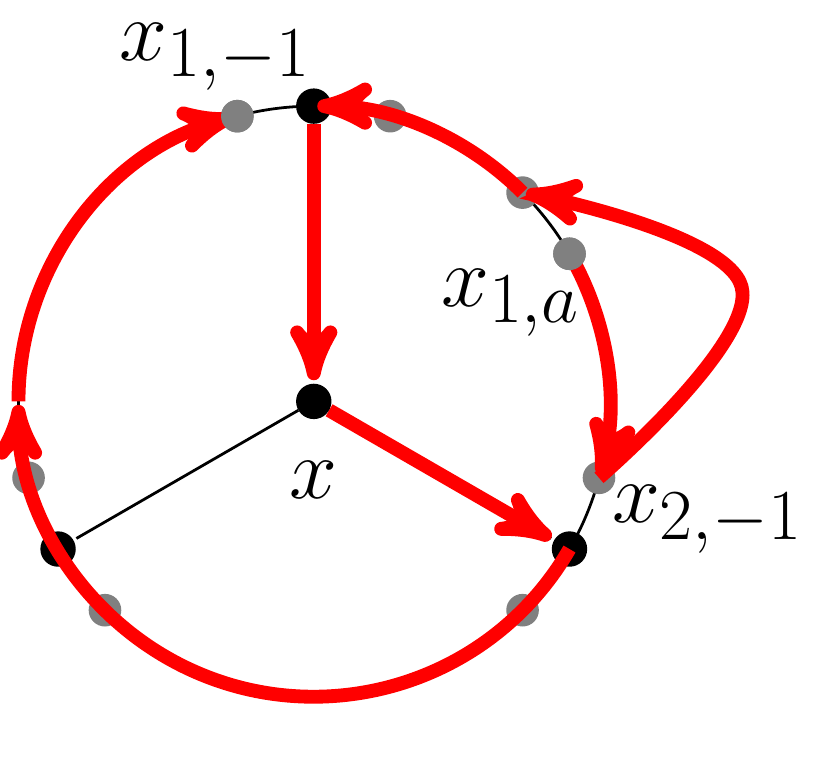}
    \caption{The path $P$ from $x_{1,a}$ to $x_{1,-1}$.}
    \label{fig:5_1}
\end{figure}

Finally, we argue $x_{1,a}$ and $x_{3,-1}$ have no CON. Suppose $y$ is such a CON; then the cycle
\[
	u_{3}C[u_{3},x_{1,a-1}]x_{1,a-1}y'x_{2,-1}C^-[x_{2,-1},x_{1,a}]x_{1,a}yx_{3,-1}C^-[x_{3,-1},u_{2}]u_2F[u_{2},u_{3}]u_3
\]
is a longer cycle than $C$. So $x_{1,a}$ has no CONs with any of $x_{1,-1},x_{2,-1},x_{3,-1}$; $x_{1,a}$ and $x_{1-1}$ have no crossings, and neither do $x_{2,-1}$ and $x_{3,-1}$. This certifies that $\{x_{1,-1},x_{2,-1},x_{3,-1},x_{1,a}\}$ is a good set, a contradiction to Lemma~\ref{good}.
\end{proof}

\begin{lem}\label{super}
If $i$ has medium type, then for  $x_{i,j} \in \{x_{i,1}, \ldots, x_{i+1, -2}\}$,
\begin{enumerate}
\item[(i)] $x_{i,j}$ and $x_{i+1,1}$ have no CONs and no crossings, and
\item[(ii)] $x_{i,j}$ and $x_{i-1,1}$ have no CONs. 
\end{enumerate}
Symmetrically, $x_{i,j} \in \{x_{i,2}, \ldots, x_{i+1, -1}\}$ and $x_{i,-1}$ have no CONs and no crossings, and $x_{i,j}$ and $x_{i-1,-1}$ have no CONs.
\end{lem}

\begin{proof}
Without loss of generality, let $i=1$. Suppose $x_{1,j}$ and $x_{2,1}$ have a common neighbor $y$ (the $x_{1,-1}$ case is symmetric). By Lemma~\ref{lots}, $x_{1,1}$ and $x_{1,j+1}$ have a CON $y'$. By Lemma~\ref{path-to-crossings} and the path 
\[
	P := x_{1,j}C^-[x_{1,j},x_{1,1}]x_{1,1}y'x_{1,j+1}C[x_{1,j+1},u_2]u_2F[u_2,u_1]u_1C^-[u_1,x_{2,1}]x_{2,1},
\]
shown in Figure~\ref{fig:5_3}, $x_{1,j}$ and $x_{2,1}$ share no CONs (otherwise $x_{1,1}$ and $x_{2,1}$ share a CON) and no crossings.

\begin{figure}[h!]
    \centering
    \includegraphics[scale=.5]{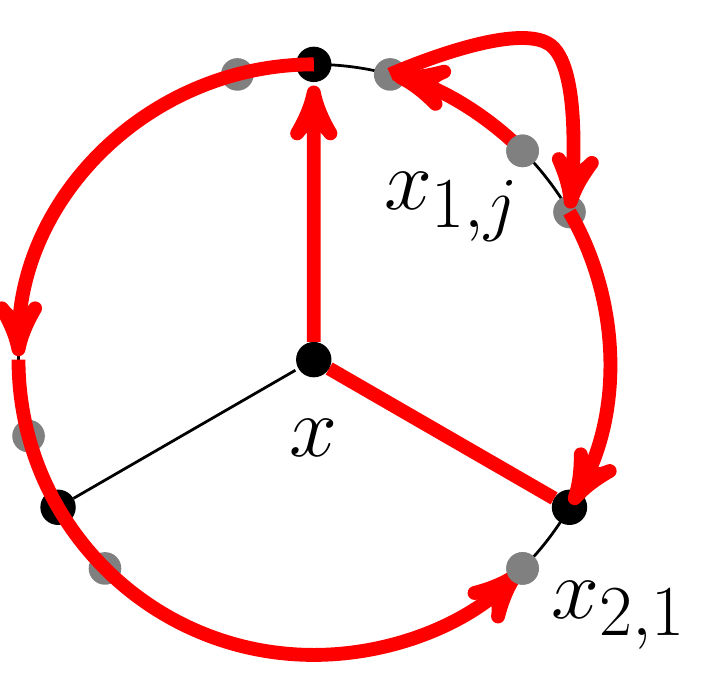}
    \caption{The path $P$ from $x_{1,j}$ to $x_{2,1}$.}
    \label{fig:5_3}
\end{figure}

Suppose that $x_{1,j}$ has a CON $y$ with $x_{3,1}$. By Lemma~\ref{lots}, $x_{1,1}$ and $x_{1,j+1}$ have a CON $y'$. Moreover, by Lemma~\ref{cross0}, $x_{1,1}$ and $x_{3,1}$ can have no CON, so $y \ne y'$. In this case, we obtain a longer cycle than $C$: the cycle
\[
	x_{1,1}C[x_{1,1}, x_{1,j}] x_{1,j} yx_{3,1} C[x_{3,1}, u_{1}] u_{1}F[u_{1} u_{3}]u_{3} C^-[u_{3}, x_{1,j+1}]x_{1,j+1}y' x_{1,1}.
\] 
This is a contradiction, so $x_{1,j}$ and $x_{3,1}$ have no CON. 
The $x_{3,-1}$ case is symmetric.
\end{proof}

%If $1, 2, 3$ all have medium type, then for each $i$, the vertices $x_{i,2} , x_{i,3}, \dots, x_{i+1,-2}$ all have a CON with both $x_{i,1}$ and with $x_{i+1,-1}$. This case will be handled later in Subsection~\ref{many-cn-section}, which yields a contradiction. Therefore, we may assume that at most two $i, j \in [3]$ have medium type.

\begin{lem}\label{zzl}
If $j\in [3]$ does not have medium type, then every $i\in [3]$ that has medium type also has long type. \end{lem}
\begin{proof}
Without loss of generality, suppose $i$ has medium type but $j=i-1$ does not. The case where $j=i+1$ is symmetric, after reorienting $C$. It suffices to show that in such a case, $i$ has long type.

By Lemma~\ref{shortsegment}, we may assume $x_{i,1} \ne x_{i+1, -1}$. Let $A = \{x_{i,2},x_{i+1,1}, x_{i,-1}, x_{i-1,1}\}$.  By Lemma~\ref{super}, $x_{i,-1}$ and $x_{i,2}$ have no CONs or crossings; by Lemma~\ref{cross0} and Lemma~\ref{nocrossings}, $x_{i-1,1}$ and $x_{i+1,1}$ have no CONs or crossings.

If $x_{i,2} \ne x_{i+1,-1}$, then Lemma~\ref{super} further tells us that $x_{i,2}$ has no CONs with $x_{i-1,1}$ or $x_{i+1,1}$. If $x_{i,2} = x_{i+1,-1}$ then Lemma~\ref{shortsegment2} gives the same conclusion. %since $i+1$ cannot have long type, $x_{i+1,-1}$ cannot have a CON with $x_{i-1,1}$, and since $i+1$ cannot have short type, $x_{i+1,-1}$ cannot have a CON with $x_{i+1,1}$.

By assumption, $i-1$ does not have medium type, so $x_{i-1,1}$ and $x_{i,-1}$ have no CONs. If $x_{i,-1}$ and $x_{i+1,1}$ also have no CONs, then $A$ is a good set, contradicting Lemma~\ref{good}. Therefore $x_{i,-1}$ and $x_{i+1,1}$ must have a CON; in other words, $i$ has long type.
\end{proof}

The three previous lemmas help us to prove the main result of this subsection:
\begin{lem}\label{nomed}
No $i \in [3]$ has medium type. \end{lem}

\begin{proof} Suppose the lemma does not hold. If all $i\in [3]$ have medium type, then by
Lemma~\ref{lots}, all of them are abundant, a contradiction to Lemma~\ref{manycon}. Thus there is a $j\in [3]$ that does 
not have medium type. Then by Lemma~\ref{zzl}, each $i\in [3]$ that has medium type also has long type. Now
Lemma~\ref{onelong} yields
 that only one $i$ can have medium type. Suppose by symmetry that this $i$ is $1$.

%Without loss of generality, let $i=1$. 
%Consider the set $\{x_{1,2},x_{2,1}, x_{1,-1}, x_{3,1}\}$. By Lemma~\ref{super} $(x_{1,2},x_{1,-1})$ have no crossings and no CONs. Additionally, $(x_{2,1},x_{3,1})$ have no crossings and no CONs. So we are either done by Lemma~\ref{good}, or one of the pairs $(x_{1,2},x_{3,1})$, $(x_{1,2},x_{2,1})$, $(x_{2,1},x_{1,-1})$, or $(x_{3,1},x_{1,-1})$ have a CON. By Lemma~\ref{super}, it cannot be $(x_{1,2},x_{3,1})$ or $(x_{1,2},x_{2,1})$. Since $3$ does not have medium type, it also cannot be $(x_{3,1},x_{1,-1})$. So $x_{2,1}$ and $x_{1,-1}$ have a CON.

Let $b$ be the smallest integer such that $x_{2,1}$ and $x_{1,-b}$ have no CON, and consider instead the set $X'=\{x_{1,2},x_{2,1},x_{3,-1},x_{1,-b}\}$. Let $y$ be the CON of $x_{1,-b+1}$ and $x_{2,1}$. By Lemma~\ref{path-to-crossings} and the path
\[
	x_{1,-b}C^-[x_{1,-b},u_3]u_3F[u_3,u_2]u_2C^-[u_2,x_{1,-b+1}]x_{1,-b+1}yx_{2,1}C[x_{2,1},x_{3,-1}]x_{3,-1}
\]
shown in Figure~\ref{fig:5_7} (left), $x_{3,-1}$ and $x_{1,-b}$ have no CON (otherwise $x_{3,-1}$ and $x_{2,1}$ have a CON, making $2$ medium-type) and can only cross at a vertex $x_{1,j}$ for $j\geq 1$ or a vertex $x_{1,-a}$ where $a<b$. Note by Lemma~\ref{cross1} they cannot cross at $u_1$.

In the first case, if $j > 1$, let $y^-=y^-(x_{1,j})$. Note that $x_{2,-1}$ and $x^-(y^-)$ share a CON $y'$. We get a contradiction by the cycle
\[
	u_3C[u_3,x^-(y^-)]x^-(y^-)y'x_{2,-1}C^-[x_{2,-1},y^-]y^-x_{3,-1}C[x_{3,-1},u_2]u_2F[u_2,u_3]u_3.
\]
If $j=1$, then let $y'$ be a CON of $x_{1,1}$ and $x_{1,2}$, and let $y$ be a CON of $x_{2,1}$ and $x_{1,-b+1}$. Then we get the longer cycle 
\[
	x_{1,1} y' x_{1,2} C[x_{1,2}, u_2] u_2 F[u_2, u_1] u_1 C^-[u_1, x_{1,-b+1}] x_{1,-b+1} y x_{2,1} C^-[x_{2,1}, x_{1,-b}] x_{1,-b} y^+(x_{1,1}) x_{1,1}.
\]
In the second case, let $y'$ be a CON of $x_{2,1}$ and $x_{1,-a}$. Then we get a longer cycle 
\[
	u_3 C[u_3, y^-(x_{1,-a})] y^-(x_{1,-a}) x_{3,-1} C^-[x_{3,-1},x_{2,1}] x_{2,1}y' x_{1,-a} C[x_{1,-a}, u_2] u_2 F[u_2, u_3] u_3.
\]

\begin{figure}[h!]
    \centering
    \includegraphics[scale=.5]{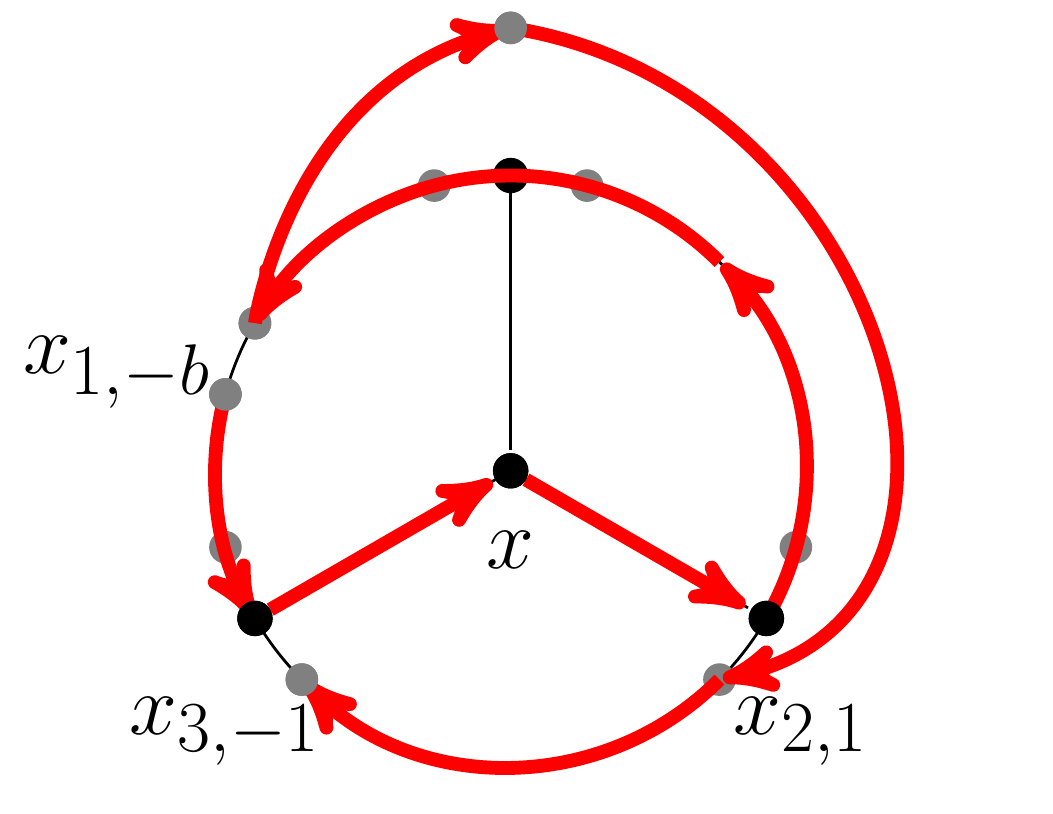}
    \includegraphics[scale=.5]{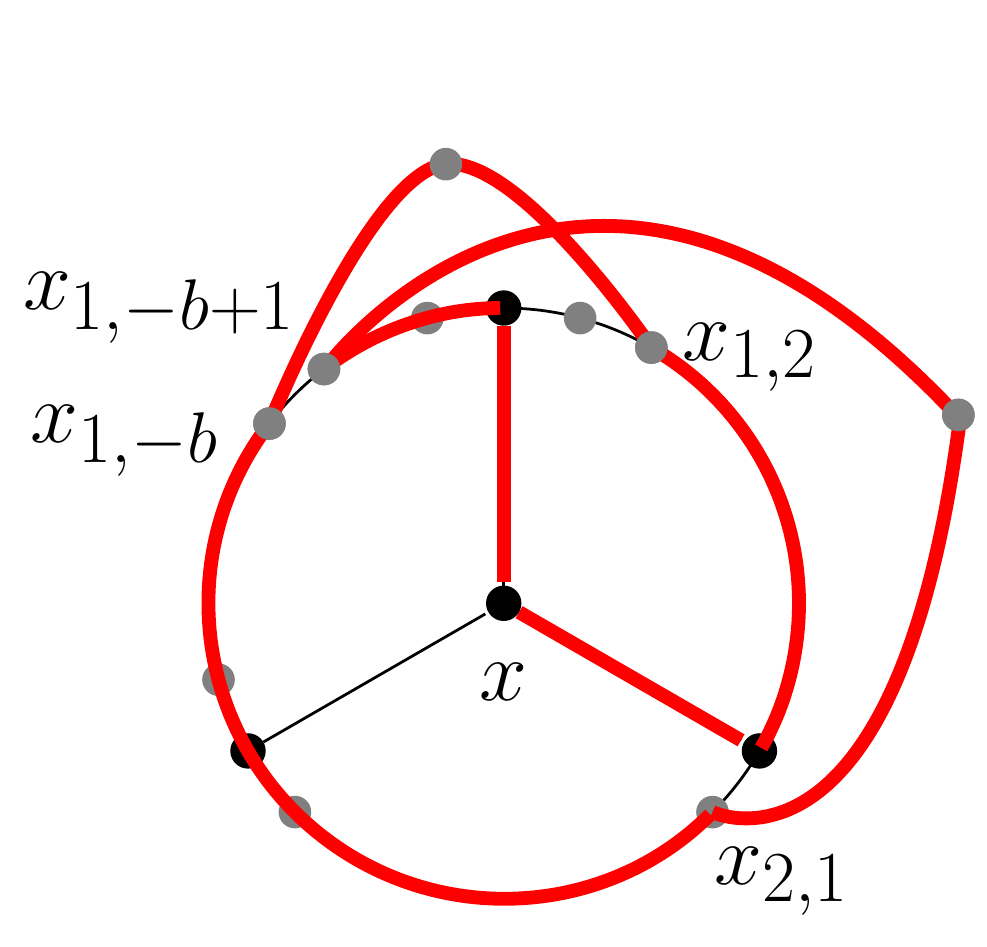}
    \caption{An $x_{1,-b}, x_{3,-1}$-path, and a longer cycle obtained when $x_{1,2}$ and $x_{1,-b}$ have a CON.}
    \label{fig:5_7}
\end{figure}

By Lemma~\ref{super}, $x_{1,2}$ and $x_{2,1}$ have no CONs and no crossings, and $x_{1,2}$ shares no CONs with $x_{3,-1}$. 

Suppose $y'$ is a CON of $x_{1,2}$ and $x_{1,-b}$. By the choice of $b$, $x_{2,1}$ and $x_{1,-b+1}$ have a CON $y$. The cycle
\[
	C' := x_{1,-b+1} y x_{2,1} C[x_{2,1}, x_{1,-b}] x_{1,-b} y' x_{1,2} C[x_{1,2}, u_2] u_2 F[u_2, u_1] u_1 C^-[u_1, x_{1,-b+1}]x_{1,-b+1}
\]
shown in Figure~\ref{fig:5_7}(right) excludes $x_{1,1}$ but contains the rest of $X \cap V(C) - \{x_{1,1}\}$. Moreover, $C'$ contains all but at most four vertices in $Y \cap C$: $y^+(x_{1,-b}), y^+(x_{1,1}),$ and possibly $y^-(x_{1,1})$ or $y^-(x_{2,1})$, if $u_1\in X$ or $u_2 \in X$ respectively. If $D$ is 2-rich, then $|C'| > |C|$, so we may assume that $D$ is not 2-rich, and $d_C(x) \leq 1$ by Lemma~\ref{t3r}. By the choice of $(C, x, F)$ as a best triple, $d_{C'}(x_{1,1}) \leq 1$ as well. Then $d_C(x) + d_{C}(x_{1,1}) \leq 1 + 1 + 4$. This contradicts Lemma~\ref{sum8}, which shows that $x_{1,2}$ and $x_{1,-b}$ share no CONs.

%If $x_{1,3}$ exists, i.e., if $x_{1,2} \neq x_{2,-1}$, then $x_{1,1}$ has a CON, $y''$ with $x_{1,3}$ by Lemma~\ref{lots}. Note that $y\neq y''$ because $x_{2,1}$ and $x_{1,1}$ share no CONs, and  $y\neq y'$ because $x_{2,1}$ and $x_{1,-b}$ share no CONs. If $y'=y''$, then the cycle $$x_{1,-b+1}C[x_{1,-b+1},u_1]u_1F[u_1,u_2]u_2C^-[u_2,x_{1,1}]x_{1,1}y'x_{1,-b}C^-[x_{1,-b},x_{2,1}]x_{2,1}yx_{1,-b+1}$$ is longer than $C$, a contradiction. If $y'\neq y''$, then the cycle $$x_{1,-b+1}C[x_{1,-b+1},u_1]u_1F[u_1,u_2]u_2C^-[u_2,x_{1,3}]x_{1,3}y''x_{1,1}C[x_{1,1},x_{1,2}]x_{1,2}y'x_{1,-b}C^-[x_{1,-b},x_{2,1}]x_{2,1}yx_{1,-b+1}$$ is longer than $C$, a contradiction. So $x_{1,2}$ shares no CONs with $x_{1,-b}$. 

Since $2$ does not have medium type, $x_{2,1}$ and $x_{3,-1}$ share no CONs. By the definition of $b$, $x_{1,-b}$ and $x_{2,1}$ share no CONs. Thus, $X'$ is good, a  contradiction to Lemma~\ref{good}.
\end{proof}

%\begin{lem}
%If $i$ has medium type for two $i\in [3]$, then $|Y|\geq 4\delta -7$.
%\end{lem}
%\begin{proof}
%Without loss of generality, let $1,2$ have medium type. Consider $X'=\{x_{1,2}, x_{1,-1}, x_{3,1}, x_{2,1}\}$.
%
%From Lemma~\ref{super}, $x_{1,2}$ and $x_{1,-1}$ have no common neighbors outside $C$ and no crossings. Additionally, $x_{3,1}$ and $x_{2,1}$ have no common neighbors outside $C$ and no crossings. Also, none of the pairs $(x_{1,2},x_{3,1})$, $(x_{2,1},x_{3,1})$ have common neighbors outside $C$. 
%
%If $x_{1,2}$ and $x_{2,1}$ had a CON $y$, then let $y'$ be the CON of $x_{1,1}$ and $x_{1,3}$ guaranteed by Lemma~\ref{lots}; since $x_{1,1}$ and $x_{2,1}$ can have no CON by Lemma~\ref{cross0}, $y \ne y$. Then the cycle
%\[
%	x_{1,1} x_{1,2} y  x_{2,1} C[x_{2,1}, u_{1}] F(u_{1}, u_{2}) C^-[u_{2}, x_{1,3}] y' x_{1,1}
%\]
%is a longer cycle than $C$, leading to a contradiction; therefore $x_{1,2}$ and $x_{2,1}$ have no CON. Lastly, $x_{3,1}$ and $x_{1,-1}$ have no CON, since in that case $3$ would also have medium type. So $X'$ is good.
%\end{proof}
%

\subsection{Eliminating long-type configurations}
\label{one-long-section}
\begin{lem}\label{nolong}
No $i \in [3]$ has long type. \end{lem}

\begin{proof} Suppose some $i \in [3]$ has long type. By Lemma~\ref{onelong},
 there is only one such $i$. By symmetry, assume
 $x_{3,-1}$ and $x_{1,1}$ have a CON $a$, i.e., only $3$ has long type. 
 Then by  Lemma~\ref{type-lemma}, since no $j$ has medium type, $1$ has short type, which means $x_{1,-1}$ and $x_{1,1}$  have a CON $b$.

Let $W = \{x_{1,-1}, x_{1,2}, x_{2,1}, x_{3,1}\}$. We will show that $W$ is a good set.

By Lemma~\ref{cross0} and Lemma~\ref{nocrossings},  $x_{2,1}$ and $x_{3,1}$ have no CON or crossings. Also, $x_{1,-1}$ and $x_{1,2}$ have no CON or crossings: This follows from Lemma~\ref{path-to-crossings}, as shown on the left in Figure~\ref{fig:s22}, where the path 
\[
	P := x_{1,-1}C^-[x_{1,-1},u_3]u_3F[u_3,u_1]u_1C[u_1,x_{1,1}]x_{1,1}ax_{3,-1}C^-[x_{3,-1},x_{1,2}]x_{1,2}
\]
agrees with the cycle $C$ on all edges.

We now show that the remaining  pairs in $W$ do not have CONs.
If $x_{1,-1}$ and $x_{2,1}$ have a CON, then we have a second long-type configuration. If $x_{1,-1}$ and $x_{3,1}$ have a CON, 
the we have a medium-type configuration.

If $x_{1,2}$ and $x_{2,1}$ have a CON $c$, then the cycle 
\[
	u_3C[u_3,x_{1,1}]x_{1,1}ax_{3,-1}C^-[x_{3,-1},x_{2,1}]x_{2,1}cx_{1,2}C[x_{1,2},u_2]u_2F[u_2,u_3]u_3
\]
 is longer than $C$, as shown in the middle of Figure~\ref{fig:s22}. Finally, if $x_{1,2}$ and $x_{3,1}$ have a CON~$c$ then  the cycle
\[
	x_{3,1}C[x_{3,1},x_{1,-1}]x_{1,-1}bx_{1,1}C^-[x_{1,1},u_1]u_1F[u_1,u_3]u_3C^-[u_3,x_{1,2}]x_{1,2}cx_{3,1}
\]
 is longer than $C$, as shown on the right in Figure~\ref{fig:s22}.

%\end{itemize}

\begin{figure}[h!]
    \centering
    \includegraphics[scale=.5]{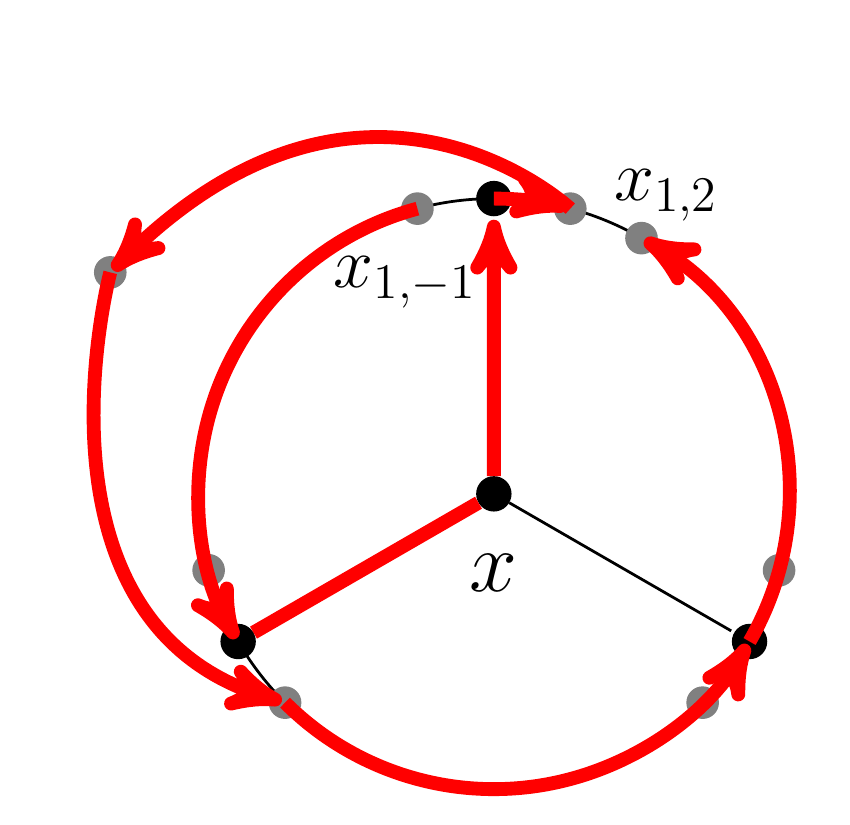}
    \includegraphics[scale=.5]{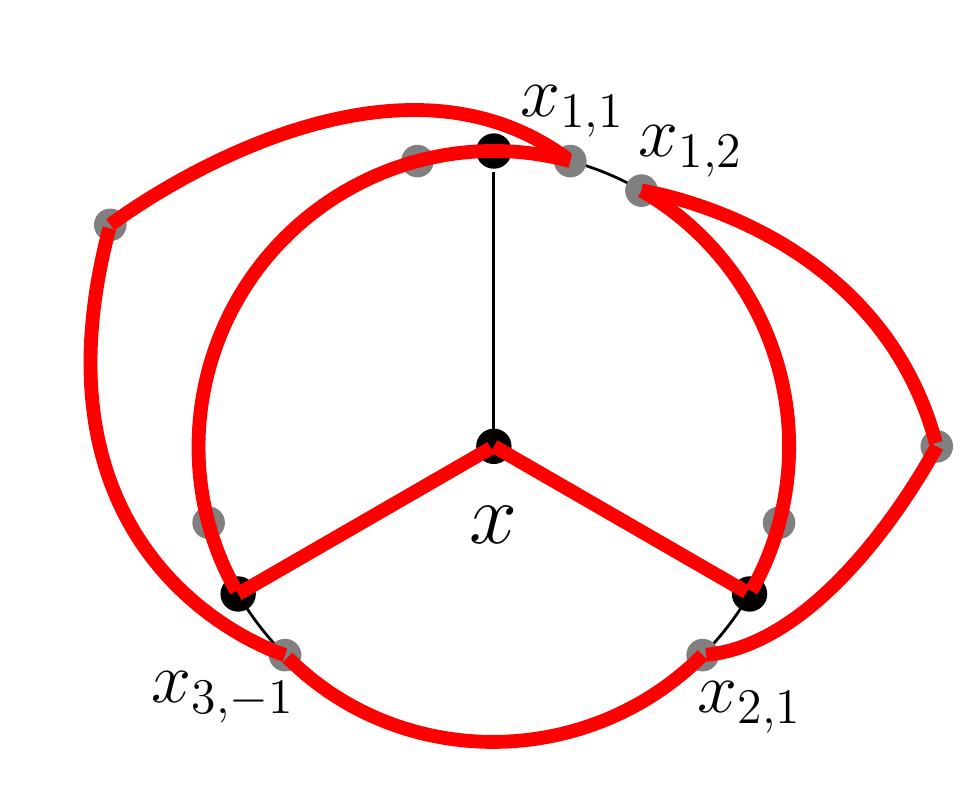}
    \includegraphics[scale=.5]{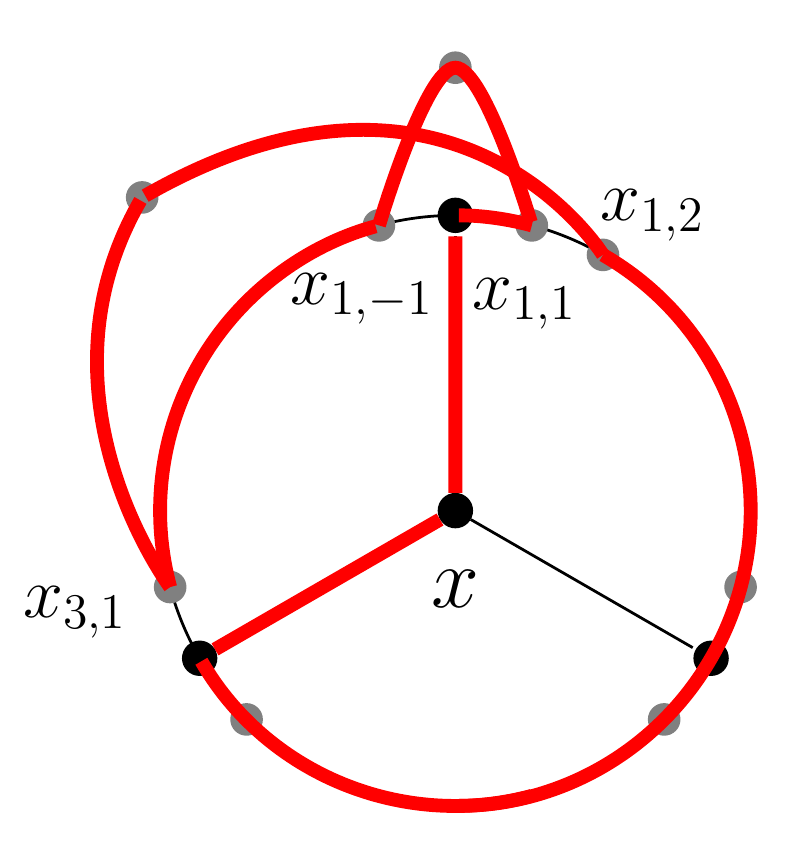}
    \caption{An $x_{1,-1}, x_{1,2}$-path, and longer cycles obtained if $x_{1,2}$ has a CON with $x_{2,1}$ or $x_{3,1}$.}
    \label{fig:s22}
\end{figure}

Therefore $W$ is a good set, contradicting Lemma~\ref{good}.
 \end{proof}

 %Thus, no $i\in[3]$ can have long type.
\subsection{Eliminating short-type configurations and finishing the proof of Theorem~\ref{jackson6}}
\label{short-section}

\begin{lem}\label{last-lemma}
If there are no long-type configurations and no medium-type configurations, then every $i\in [3]$ is abundant.
\end{lem}
\begin{proof}
By Lemma~\ref{type-lemma}, every $i\in [3]$ has short type and no other types. 
%In this subsection, we assume that all three short-type configurations, and none of the long-type configurations, exist.

For definiteness, consider $i=1$. By the definition of short type, $x_{1,-1}$ and $x_{1,1}$ have a CON. Let $b>1$ be the least integer such that $x_{1,-1}$ has no CON with $x_{1,b}$. Some such $b$  exists, because $x_{2,-1}$ has no CON with $x_{1,-1}$. Moreover, if  $x_{1,b} = x_{2,-1}$, then we  find a cycle $C'$ longer than $C$: if $y_1$ is a CON of $x_{1,-1}$ and $x_{2,-2}$, and $y_2$ is a CON of $x_{2,-1}$ and $x_{2,1}$, then $y_1 \ne y_2$ (since $x_{1,-1}$ has no CON with $x_{2,-1}$) and therefore
\[
u_1 C[u_1, x_{2,-2}] x_{2,-2}y_1 x_{1,-1} C^-[x_{1,-1}, x_{2,1}]x_{2,1} y_2 x_{2,-1} C[x_{2,-1}, u_2]u_2 F[u_2, u_1] u_1 
\]
is a  cycle longer than $C$. So $b$ exists and $x_{1,b} \ne x_{2,-1}$. Note that this implies $x_{1,2}\neq x_{2,-1}$.

Consider the set $W_b = \{x_{1,-1}, x_{1,b}, x_{2,-1}, x_{3,1}\}$. We will show that it is {\em almost} a good set.
%\begin{itemize}
%\item $x_{1,-1}$ and $x_{1,b}$ have no CON by our choice of $b$.
%\item $x_{1,-1}$ and $x_{2,-1}$ have no CON or crossings, by Lemma~\ref{cross0}.
%\item $x_{1,-1}$ and $x_{3,1}$ have no CON. Otherwise, a medium-type configuration would be formed.
%\item $x_{1,b}$ and $x_{3,1}$ cannot share the CON of $x_{1,-1}$ and $x_{1,b-1}$ because $x_{1,-1}$ and $x_{3,1}$ have no CON. Let $c$ be the CON of $x_{1,-1}$ and $x_{1,b-1}$. By applying Lemma~\ref{path-to-crossings} to the path\\ $x_{3,1}C[x_{3,1},x_{1,-1}]x_{1,-1}cx_{1,b-1}C^-[x_{1,b-1},u_1]u_1F[u_1,u_3]u_3C[u_3,x_{1,b}]x_{1,b}$, as on the left in Figure~\ref{fig:s23}, we see that they can have no other CON, and can only cross at a vertex $x_{1,a}$ with $a<b$.
%
%If such a crossing existed, however, then in particular $x_{3,1}$ would be adjacent to a neighbor of $x_{1,a}$ and letting $c$ be the CON of $x_{1,-1}$ and $x_{1,a+1}$ we would obtain a longer cycle $x_{3,1}C[x_{3,1},x_{1,-1}]x_{1,-1}cx_{1,a+1}C[x_{1,a+1},u_3]u_3F[u_3,u_1]u_1C[u_1,y^-(x_{1,a+1})]y^-(x_{1,a+1})x_{3,1}$, as shown in the right of Figure~\ref{fig:s23}. In the special case $a = b-1$, the cycle looks only slightly different (letting $c$ be the CON of $x_{1,-1}$ and $x_{1,b-1}$ it is\\ $x_{3,1}C[x_{3,1},x_{1,-1}]x_{1,-1}cx_{1,b-1}C[x_{1,b-1},u_3]u_3F[u_3,u_1]u_1C[u_1,y^+(x_{1,b-1})]y^+(x_{1,b-1})x_{3,1}$.
%
%We conclude that $x_{1,b}$ and $x_{3,1}$ have no CON or crossings.
%
%\item $x_{2,-1}$ and $x_{3,1}$ have no CON. Otherwise, a long-type configuration would be formed.
%\end{itemize}

By Lemma~\ref{cross0} and Lemma~\ref{nocrossings},  $x_{1,-1}$ and $x_{2,-1}$ have no CON or crossing. 
A CON of $x_{1,b}$ and $x_{3,1}$ is distinct from any CON of $x_{1,-1}$ and $x_{1,b-1}$ because $x_{1,-1}$ and $x_{3,1}$ have no CON. 
Let $c$ be the CON of $x_{1,-1}$ and $x_{1,b-1}$. By applying Lemma~\ref{path-to-crossings} to the path 
\[
	x_{3,1}C[x_{3,1},x_{1,-1}]x_{1,-1}cx_{1,b-1}C^-[x_{1,b-1},u_1]u_1F[u_1,u_3]u_3C^-[u_3,x_{1,b}]x_{1,b},
\]
as on the left in Figure~\ref{fig:s23}, we see that they can have no other CON, and can only cross at a vertex $x_{1,a}$ with $a<b$.

If such a crossing existed, however, then in particular $x_{3,1}$ would be adjacent to a neighbor of $x_{1,a}$ and letting $c$ be the CON of $x_{1,-1}$ and $x_{1,a+1}$ we would obtain a longer cycle 
\[
	x_{3,1}C[x_{3,1},x_{1,-1}]x_{1,-1}cx_{1,a+1}C[x_{1,a+1},u_3]u_3F[u_3,u_1]u_1C[u_1,y^-(x_{1,a+1})]y^-(x_{1,a+1})x_{3,1}
\] 
as shown on the right of Figure~\ref{fig:s23}. In the special case $a = b-1$, the cycle looks only slightly different. Letting $c$ be the CON of $x_{1,-1}$ and $x_{1,b-1}$, it is 
\[
	x_{3,1}C[x_{3,1},x_{1,-1}]x_{1,-1}cx_{1,b-1}C^-[x_{1,b-1},u_1]u_1F[u_1,u_3]u_3C^-[u_3,y^+(x_{1,b-1})]y^+(x_{1,b-1})x_{3,1}.
\]

We conclude that $x_{1,b}$ and $x_{3,1}$ have no CON or crossings.

By the choice of $b$, $x_{1,-1}$ and $x_{1,b}$ have no CON.
 The pair $x_{1,-1}$ and $x_{3,1}$ have no CON, otherwise a medium-type configuration would be formed. The pair $x_{2,-1}$ and $x_{3,1}$ have no CON, otherwise a long-type configuration would be formed.

\begin{figure}[h!]
    \centering
    \includegraphics[scale=.5]{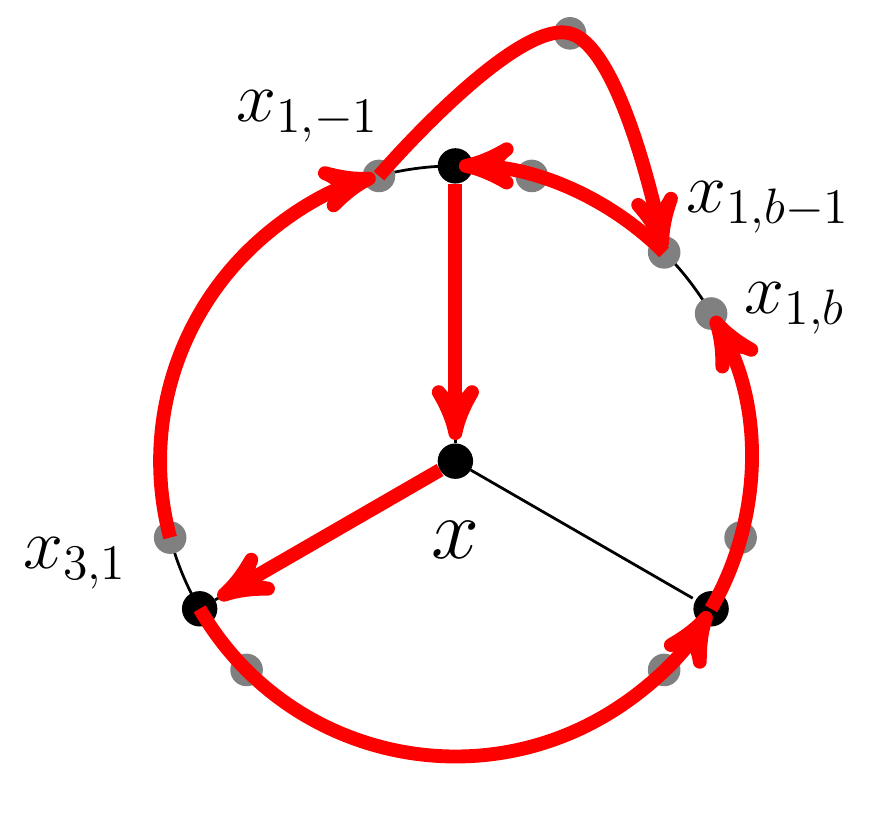}
    \includegraphics[scale=.5]{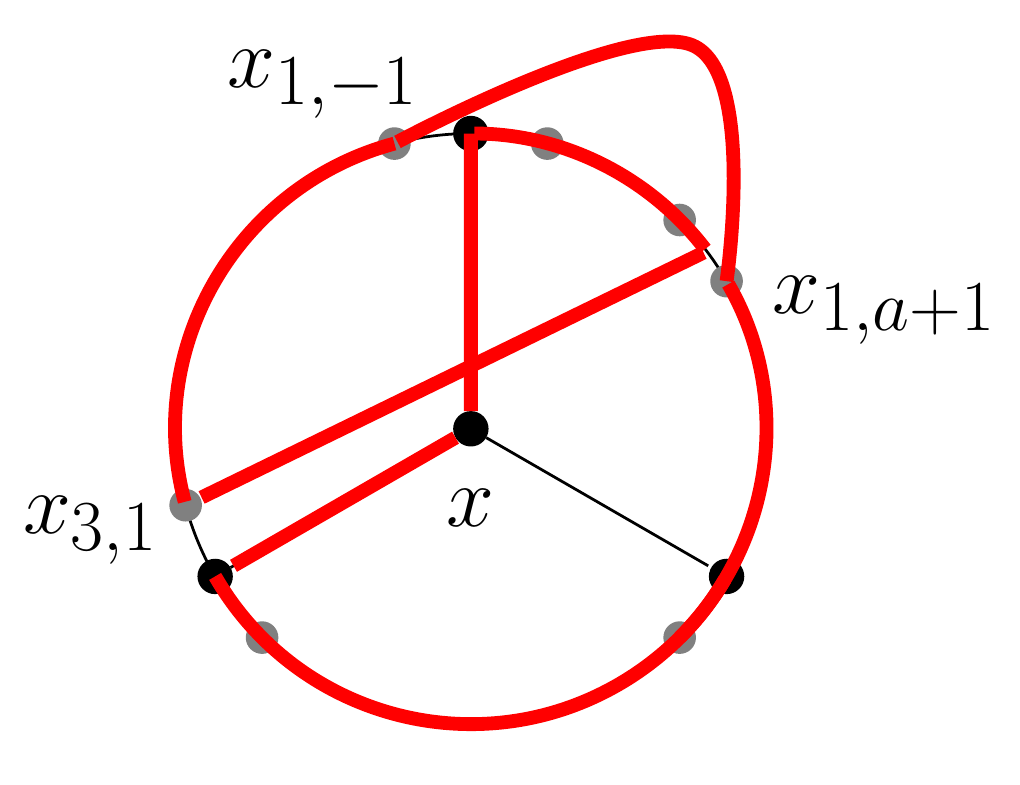}
    \caption{An $x_{3,1},x_{1,b}$-path, and a longer cycle obtained if $x_{1,b}$ and $x_{3,1}$ have a crossing at $x_{1,a}$.}
    \label{fig:s23}
\end{figure}

If  $x_{1,b}$ and $x_{2,-1}$ have no CON, then $W_b$ is a good set, a contradiction to Lemma~\ref{good}.
Thus, $x_{1,b}$ and $x_{2,-1}$ have a CON. 

We now prove that
\begin{equation}\label{last}
\mbox{\em for each $c\geq b$ such that $x_{1,c}\in C[x_{1,b},x_{2,-2}]$, vertices $x_{1,c}$ and $x_{2,-1}$ have a CON.}
\end{equation}
Indeed, suppose~\eqref{last} does not hold and
$c$ is the least  integer such that $c\geq b$ and $x_{1,c}$ has no CON with $x_{2,-1}$. By the previous paragraph, $c>b$.
Consider the set $W_c = \{x_{1,-1}, x_{1,c}, x_{2,-1}, x_{3,-1}\}$. We will show that this is a good set.

Indeed, $x_{1,-1}$ and $x_{2,-1}$ have no CON or crossing, by Lemma~\ref{cross0} and Lemma~\ref{nocrossings}. Any CON of $x_{1,c}$ and $x_{3,-1}$ is distinct from any  CON of $x_{1,c-1}$ and $x_{2,-1}$, since $x_{3,-1}$ and $x_{2,-1}$ have no CON. They have no other CON or crossings, as shown by the path 
\[
	x_{1,c}C[x_{1,c},x_{2,-1}]x_{2,-1}qx_{1,c-1}C^-[x_{1,c-1},u_3]u_3F[u_3,u_2]u_2C[u_2,x_{3,-1}] x_{3,-1}
\] (see the left in Figure~\ref{fig:s24}) and Lemma~\ref{path-to-crossings}, where $q$ is the CON of $x_{2,-1}$ and $x_{1,c-1}$.

We show that the remaining pairs have no CONs. Indeed, $x_{1,c}$ and $x_{2,-1}$ have no CON by our choice of $c$. The pairs $\{x_{1,-1},x_{3,-1}\}$ and $\{x_{2,-1},x_{3,-1}\}$ have no CONs, by Lemma~\ref{cross0}. Finally, suppose $r$ is a CON of $x_{1,-1}$ and $x_{1,c}$ Let $q$ be a CON of $x_{2,-1}$ and $x_{1,c-1}$. Then  the cycle 
\[
	u_2C[u_2,x_{1,-1}]x_{1,-1}rx_{1,c}C[x_{1,c},x_{2,-1}]x_{2,-1}qx_{1,c-1}C^-[x_{1,c-1},u_1]u_1F[u_1,u_2]u_2
\] shown on the right of Figure~\ref{fig:s24} is longer than $C$.

%
%\begin{itemize}
%\item $x_{1,-1}$ and $x_{1,c}$ have no CON. If $q$ were a CON of $x_{2,-1}$ and $x_{1,c-1}$, let $r$ be the CON of $x_{1,-1}$ and $x_{1,c}$. Then  the cycle\\ $u_2C[u_2,x_{1,-1}]x_{1,-1}rx_{1,c}C[x_{1,c},x_{2,-1}]x_{2,-1}qx_{1,c-1}C^-[x_{1,c-1},u_1]u_1F[u_1,u_2]u_2$ shown on the left of Figure~\ref{fig:s24} is longer than $C$.
%\item $x_{1,-1}$ and $x_{2,-1}$ have no CON or crossings, by Lemma~\ref{cross0}.
%\item $x_{1,-1}$ and $x_{3,-1}$ have no CON, by Lemma~\ref{cross0}.
%\item $x_{1,c}$ and $x_{2,-1}$ have no CON by our choice of $c$.
%\item $x_{1,c}$ and $x_{3,-1}$ cannot share the CON of $x_{1,c-1}$ and $x_{2,-1}$, since $x_{3,-1}$ and $x_{2,-1}$ have no CON. They have no other CON or crossings, as shown by the path\\ $x_{1,c}C[x_{1,c},x_{2,-1}]x_{2,-1}qx_{1,c-1}C^-[x_{1,c-1},u_3]u_3F[u_3,u_2]u_2C[u_2,x_{3,-1}]$ (on the right in Figure~\ref{fig:s24}) and Lemma~\ref{path-to-crossings}, where $q$ is the CON of $x_{2,-1}$ and $x_{1,c-1}$.
%\item $x_{2,-1}$ and $x_{3,-1}$ have no CON, by Lemma~\ref{cross0}.
%\end{itemize}
\begin{figure}[h!]
    \centering
    \includegraphics[scale=.5]{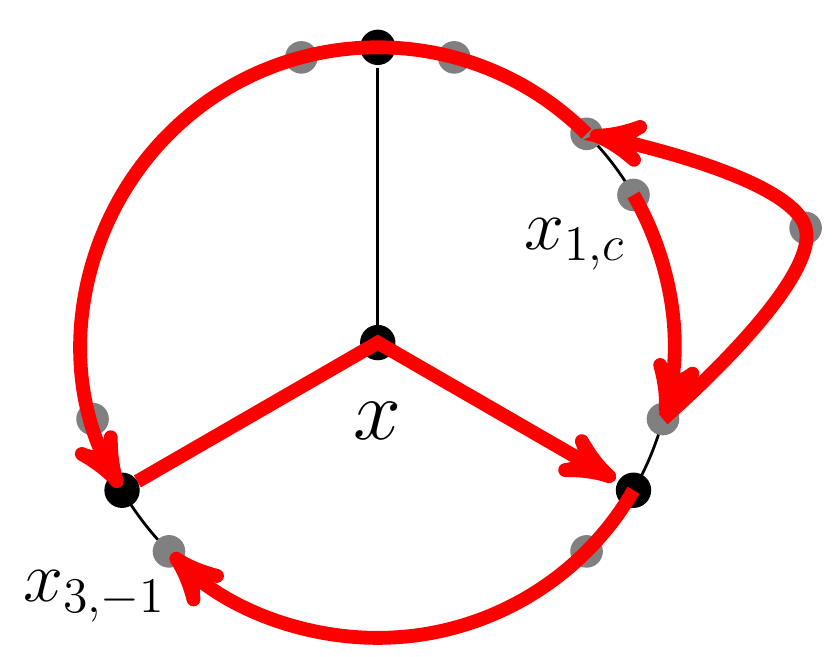}
    \includegraphics[scale=.5]{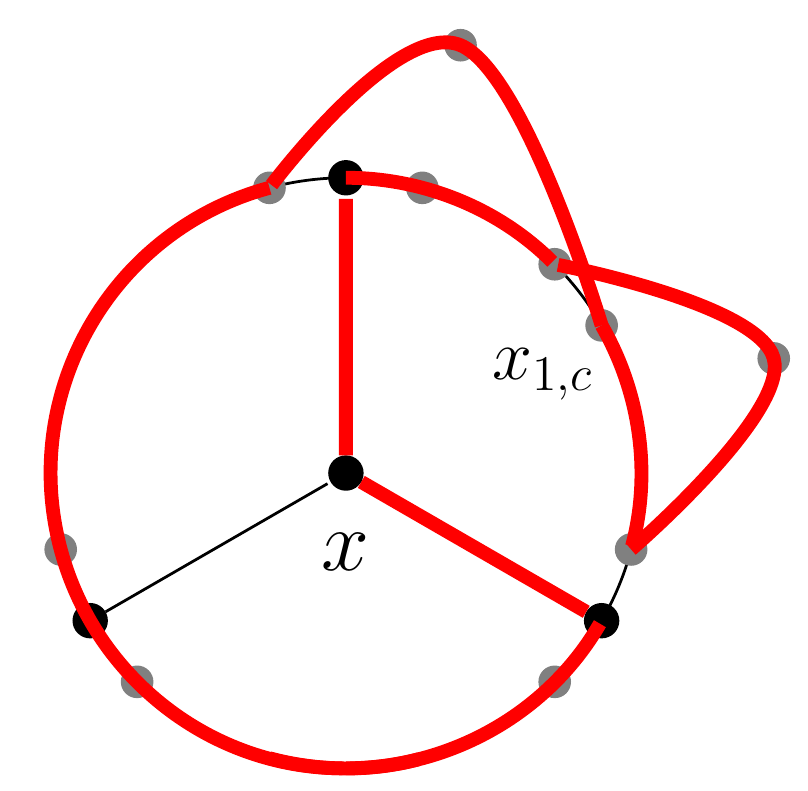}
    \caption{An $x_{1,c},x_{3,-1}$-path, and a longer cycle obtained when $x_{1,-1}$ and $x_{1,c}$ have a CON.}
    \label{fig:s24}
\end{figure}

Therefore we have a good set of size 4, a contradiction to Lemma~\ref{good}. This proves~\eqref{last}. In other words, $x_{1,1}, x_{1,2}, \dots, x_{1,b-1}$ all have a CON with $x_{1,-1}$ while $x_{1,b}, x_{1,b+1}, \dots, x_{2,-2}$ all have a CON with $x_{2,-1}$. Moreover, in this case, $x_{2,1}$ and $x_{2,-2}$ can have no CON, or else we obtain a longer cycle, 
\[
	x_{2,1}C[x_{2,1},x_{1,-1}]x_{1,-1}rx_{1,b-1}C^-[x_{1,b-1},u_1]u_1F[u_1,u_2]u_2C^-[u_2,x_{2,-1}]x_{2,-1}sx_{1,b}C[x_{1,b},x_{2,-2}]x_{2,-2}tx_{2,1},
\] where $r$ is the CON of $x_{1,-1}$ and $x_{1,b-1}$, $s$ is the CON of $x_{2,-1}$ and $x_{1,b}$, and $t$ is the CON of $x_{2,-2}$ and $x_{2,1}$, as shown in Figure~\ref{fig:xc}.
\begin{figure}[h!]
    \centering
    \includegraphics[scale=.5]{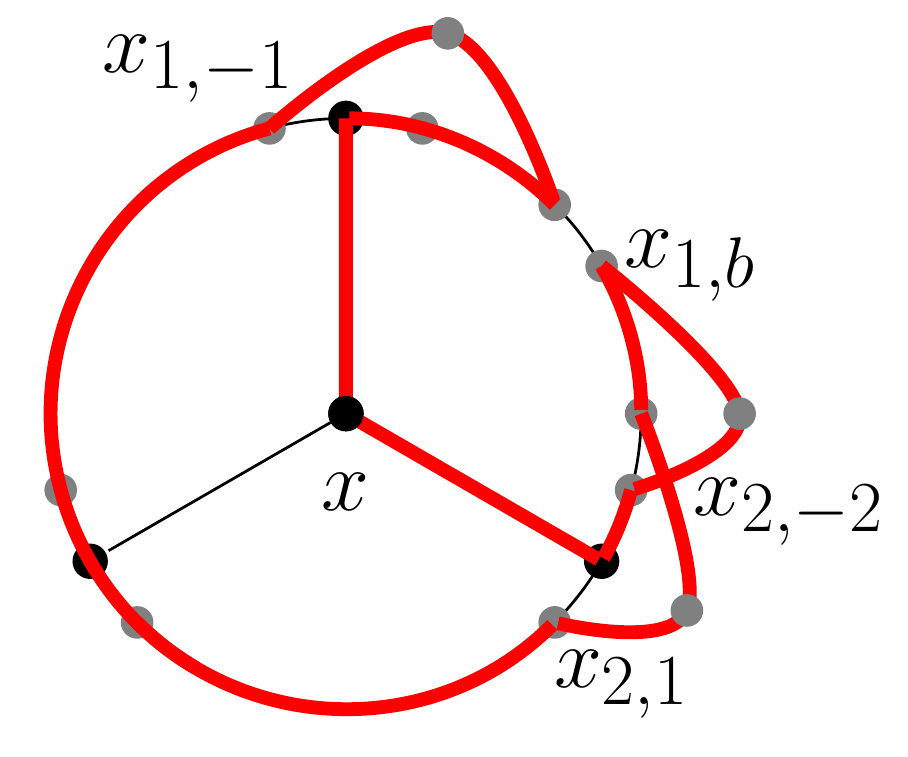}
    \caption{A longer cycle obtained when $x_{2,1}$ and $x_{2,-2}$ have a CON.}
    \label{fig:xc}
\end{figure}

We can apply the argument in this subsection in six possible ways: we can swap the roles of $x_{1,1}$ and $x_{1,-1}$ in the argument above, and we can choose any of the three short-type configurations in place of the one formed by $x_{1,1}$ and $x_{1,-1}$. All six of these arguments must terminate in the same case. In particular, just as we concluded that $x_{2,1}$ and $x_{2,-2}$ can have no CON, we also conclude that $x_{1,-1}$ and $x_{1,2}$ can have no CON. This means that in the argument above (and in all variations of the argument), we must have $b=2$.

Therefore, for each $i$, the vertices $x_{i,2}, x_{i,3}, \dots, x_{i+1,-2}$ all have a CON with both $x_{i,1}$ and with $x_{i+1, -1}$. 
In other words, all $i\in [3]$ are abundant.
\end{proof}

By Lemma~\ref{nomed} and Lemma~\ref{nolong}, no $i \in [3]$ has medium or long type. Therefore by Lemma~\ref{last-lemma}, every $i \in [3]$ is abundant. This contradicts Lemma~\ref{manycon}, completing the proof of Theorem~\ref{jackson6}.

\section{Concluding remarks}
\begin{enumerate}
\item Theorem~\ref{jackson6} is a natural $3$-connected strengthening of Conjecture~\ref{jacksonconj} for $2$-connected graphs. Consider the following family of $k$-connected graphs.

\begin{const}\label{con_q}
Let $k$ be a positive integer, and let $n_1 \geq \ldots \geq n_{k+1} \geq 1$ be such that $n_1 + \ldots + n_{k+1} = n$.
 Let $G_k(n_1, \ldots, n_{k+1}; \delta) \in \cG(n,(k+1)(\delta - k)+k, \delta)$ be the bipartite graph obtained from $K_{\delta-k, n_1} \cup \ldots  \cup K_{\delta-k, n_{k+1}} $  by adding $k$ vertices $a_1, \ldots, a_{k}$ that are each adjacent to every vertex in the parts of size $n_1, \ldots, n_{k+1}$. Let $\mathcal G_k(n, \delta)$ be the collection of the graphs $G_k(n_1, \ldots, n_{k+1}; \delta)$ for all suitable choices of $n_1, \ldots, n_{k+1}$.
 \end{const}
 
When $k=2$ or $k=3$, $\mathcal G_k$ is the family of all graphs in Construction~\ref{con5} or Construction~\ref{con6} respectively.
 
\begin{question}
Let $m,n,k,\delta$ be integers. Suppose $k \geq 4$,  $\delta\geq n$ and $m \leq (k+1)(\delta - k)+k-1$. Is it true that every $k$-connected graph $G \in \cG(n,m,\delta)$ contains a cycle of length $2n$? Moreover, if $k \geq 3$, are the graphs in the family $\mathcal G_k(n, \delta)$ the only extremal examples with $m = (k+1)(\delta - k)+k$?
\end{question}

If the answer is negative, it would also be interesting to find the value(s) of $k$ at which other extremal examples occur.
\item Jackson also made the following conjecture.
\begin{conj}[Jackson~\cite{jackson}]\label{jacksonconj2} Let $m,n,\delta$ be integers with $n > \delta$. If a graph $G \in \cG(n,m,\delta)$ is 2-connected and satisfies 
\[
	m\leq \left\lfloor \frac{2(n-\alpha)}{\delta-1-\alpha} \right\rfloor (\delta-2) + 1
\]
where $\alpha = 1$ if $\delta$ is even and $\alpha = 0$ if $\delta$ is odd, then $G$ contains a cycle of length at least $2\min(n,\delta)$.\end{conj} 
This conjecture remains open. A weaker version is proved in~\cite{KL} in the language of hypergraphs.

\end{enumerate}

%\medskip
%{\bf Acknowledgement.}

\end{document}